\documentclass{amsart}
\usepackage{amssymb}
\usepackage{amsmath}
\usepackage{amsthm}
\usepackage{amsfonts}
\usepackage{amsxtra}
\usepackage[all]{xy}
\usepackage{verbatim}
\usepackage{color}
\usepackage[normalem]{ulem}
\usepackage[usenames,dvipsnames]{xcolor}
\usepackage{graphicx}							
\usepackage{tikz}
\usetikzlibrary{matrix,arrows}

\newcommand{\diag}{\mathrm{diag}}
\newcommand{\symcofree}{S^c}
\newcommand{\cofree}{T^c}

\newcommand{\coTHH}{\mathrm{coTHH}}
\newcommand{\THH}{\mathrm{THH}}
\newcommand{\sma}{\wedge}

\newcommand{\coker}{\mathrm{coker}\:}
\newcommand{\eq}{\mathrm{eq}}
\newcommand{\res}{\mathrm{res}}
\newcommand{\coHH}{\mathrm{coHH}}
\newcommand{\HH}{\mathrm{HH}}
\newcommand{\cotor}{\mathrm{Cotor}}
\newcommand{\tor}{\mathrm{Tor}}
\newcommand{\holim}{\mathrm{holim}}
\newcommand{\colim}{\operatorname*{colim}}
\DeclareMathOperator{\Tot}{Tot}

\providecommand{\abs}[1]{\lvert#1\rvert}

\newcommand{\Spec}{\mathrm{Spec}}
\newcommand{\dgm}{\mathrm{dgmod}}

\newcommand{\D}{\mathcal{D}}
\newcommand{\coalgD}{\mathrm{Coalg}_{\D}}
\newcommand{\comcoalgD}{\mathrm{comCoalg}_{\D}}
\newcommand{\coalgDnil}{\mathrm{Coalg}_{\D}^{\mathrm{conil}}}
\newcommand{\comcoalgDnil}{\mathrm{comCoalg}_{\D}^{\mathrm{conil}}}
\newcommand{\comcoalg}{\mathrm{comCoalg}}
\newcommand{\Set}{\mathrm{Set}}
\newcommand{\sSet}{\mathrm{sSet}}
\newcommand{\ckmod}{\mathrm{ckmod}}
\newcommand{\skmod}{\mathrm{skmod}}

\newcommand{\cha}{\mathrm{char}}
\newcommand{\op}{\mathrm{op}}
\newcommand{\Hom}{\mathrm{Hom}}
\newcommand{\one}{1}
\newcommand{\ra}{\rightarrow}
\def\ie{\emph{i.e.\,\,}}
\def\eg{\emph{e.g.\,\,}}
\newcommand{\F}{\mathbb{F}}
\newcommand{\inv}{{-1}}
\newcommand{\xto}[1]{\xrightarrow{#1}}
\newcommand{\fr}{\mathrm{fr}}
\newcommand{\sh}{\mathrm{sh}}
\newcommand{\sgn}{\mathrm{sgn}}
\newcommand{\sphere}{\mathbb{S}}
\newcommand{\comult}{\triangle}

\newcommand{\sk}{\mathrm{sk}}
\newcommand{\Map}{\mathrm{Map}}
\newcommand{\id}{\mathrm{id}}

\theoremstyle{plain}   
\newtheorem{thm}{Theorem}[section] 
\newtheorem{cor}[thm]{Corollary}
\newtheorem{lemma}[thm]{Lemma}
\newtheorem{prop}[thm]{Proposition}

\newtheorem{thrm}[thm]{Theorem}

\theoremstyle{definition}
\newtheorem{defn}[thm]{Definition}

\theoremstyle{remark}
\newtheorem{rem}[thm]{Remark}
\newtheorem{ex}[thm]{Example}

\let\oldmarginpar\marginpar
\renewcommand\marginpar[1]{\-\oldmarginpar[\raggedleft\footnotesize #1]%
{\raggedright\footnotesize #1}}

\makeatletter
\newcommand*{\relrelbarsep}{.386ex}
\newcommand*{\relrelbar}{%
  \mathrel{%
    \mathpalette\@relrelbar\relrelbarsep
  }%
}
\newcommand*{\@relrelbar}[2]{%
  \raise#2\hbox to 0pt{$\m@th#1\relbar$\hss}%
  \lower#2\hbox{$\m@th#1\relbar$}%
}
\providecommand*{\rightrightarrowsfill@}{%
  \arrowfill@\relrelbar\relrelbar\rightrightarrows
}
\providecommand*{\leftleftarrowsfill@}{%
  \arrowfill@\leftleftarrows\relrelbar\relrelbar
}
\providecommand*{\xrightrightarrows}[2][]{%
  \ext@arrow 0359\rightrightarrowsfill@{#1}{#2}%
}
\providecommand*{\xleftleftarrows}[2][]{%
  \ext@arrow 3095\leftleftarrowsfill@{#1}{#2}%
}
\makeatother

\title{Computational Tools for  Topological coHochschild Homology}
\author[Bohmann]{Anna Marie Bohmann}
\author[Gerhardt]{Teena Gerhardt}
\author[{H\o genhaven}]{Amalie {H\o genhaven}}
\author[Shipley]{Brooke Shipley}
\author[Ziegenhagen]{Stephanie Ziegenhagen}

\address{Department of Mathematics, Vanderbilt University, 1326 Stevenson Center, Nashville, TN, 37240, USA}
\email{am.bohmann@vanderbilt.edu}

\address{Department of Mathematics, Michigan State University, 619 Red Cedar Road, East Lansing, MI, 48824, USA}
\email{teena@math.msu.edu}

\address{Department of Mathematics, Copenhagen University, Universitetsparken 5, Copenhagen, Denmark}
\email{hogenhaven@math.ku.dk}

\address{Department of Mathematics, Statistics, and Computer Science, University of Illinois at
Chicago, 508 SEO m/c 249,
851 S. Morgan Street,
Chicago, IL, 60607-7045, USA}
    \email{shipleyb@uic.edu}

\address{KTH Royal Institute of Technology, Department of Mathematics, SE-100 44 Stockholm, Sweden}
\email{szie@kth.se}

\date{\today}

\keywords{Topological Hochschild homology, coalgebra, Hochschild--Kostant--Rosenberg}
\subjclass[2010]{Primary:{ 19D55, 16T15; Secondary: 16E40, 55U35, 55P43 }}

\begin{document}

\begin{abstract}
{In recent work, Hess and Shipley \cite{hs.coTHH} defined a theory of topological coHochschild homology (coTHH) for coalgebras. In this paper we develop computational tools to study this new theory. In particular, we prove a Hochschild-Kostant-Rosenberg type theorem in the cofree case for differential graded coalgebras. We also develop a coB\"okstedt spectral sequence to compute the homology of coTHH for coalgebra spectra. We use a coalgebra structure on this spectral sequence to produce several computations.}
\end{abstract}
\maketitle

 \section{Introduction}
The theory of Hochschild homology for algebras has a topological analogue, called topological Hochschild homology (THH). Topological Hochschild homology for ring spectra is defined by changing the ground ring in the ordinary Hochschild complex for rings from the integers to the sphere spectrum. For coalgebras, there is a theory dual to Hochschild homology called coHochschild homology. Variations of coHochschild homology for classical coalgebras, or corings, appear  in \cite[Section 30]{brzeznski-wisbauer}, and \cite{doi}, for instance, and the coHochschild complex for differential graded coalgebras appears in \cite{hps}. In recent work \cite{hs.coTHH}, Hess and Shipley define a topological version of coHochschild homology (coTHH), which is dual to topological Hochschild homology.

 In this paper we develop computational tools for $\coTHH$.  
We begin by giving a very general definition for $\coTHH$ in a  model category endowed with a symmetric monoidal structure in terms of the homotopy limit of a cosimplicial object.  This level of generality makes concrete calculations difficult;  nonetheless, we give more accessible descriptions of $\coTHH$ in an assortment of algebraic contexts. 

One of the starting points in understanding Hochschild homology is the Hoch\-schild\---Kostant--Rosenberg theorem, which, in its most basic form, identifies the Hochschild homology of free commutative algebras. See  \cite{hochschild-kostant-rosenberg} for the classical result and  \cite{mccarthy-minasian} for an analogue in the category of spectra. The set up for coalgebras is not as straightforward as in the algebra case and we return to the general situation to analyze the ingredients necessary for defining an appropriate notion of cofree coalgebras.  In Theorem \ref{thrm:coTHHcofreeinsymmonoidalmodelcat} we conclude that the Hochschild--Kostant--Rosenberg theorem for cofree coalgebras in an arbitrary model category boils down to an analysis of the interplay between the notion of ``cofree'' and homotopy limits.  
We then prove our first computational result,  a Hochschild--Kostant--Rosenberg theorem for  cofree differential graded coalgebras. 
A similar result has been obtained by Farinati and Solotar in \cite{farinati-solotar-ext} and \cite{farinati-solotar} in the ungraded setting.

 \begin{thrm} \label{thm:HKR} Let $X$ be a nonnegatively graded cochain complex over a field $k$.  Let $\symcofree(X)$ be the cofree coaugmented cocommutative coassociative conilpotent coalgebra over $k$ cogenerated by $X$.
Then there is a quasi-isomorphism of differential graded $k$-modules
\[
\coTHH(\symcofree(X)) \simeq \Omega^{\symcofree(X)\vert k}.\]
where the right hand side is an explicit differential graded $k$-module defined in Section \ref{sect:categoricalhochschildhomology}.
\end{thrm}
The precise definition of $\Omega^{\symcofree(X)\vert k}$ depends on the characteristic of $k$. Let $U(\symcofree(X))$ denote the underlying differential graded $k$-modules of the cofree coalgebra $\symcofree(X)$. If $\cha(k) \neq 2$, we can compute $\coTHH(\symcofree(X))$ as $U(\symcofree(X)) \otimes U(\symcofree(\Sigma^{-1} X))$, while in characteristic $2$, we have $\coTHH(\symcofree(X)) \cong U(\symcofree(X)) \otimes \Lambda (\Sigma^{-1}X)$, where $\Lambda$ denotes exterior powers. 
Definitions of all of the terms here appear in Sections \ref{sect:definitions} and \ref{sect:categoricalhochschildhomology} and this theorem is proved as Theorem \ref{thrm:HKRcofreedg}.

We then develop calculational tools to study $\coTHH(C)$ for a coalgebra spectrum $C$ in a symmetric monoidal model category of spectra; see Section~\ref{sect:thespectralsequence}. Recall that for topological Hochschild homology, the B\"okstedt spectral sequence is an essential computational tool. For a field $k$ and a ring spectrum $R$, the B\"okstedt spectral sequence is of the form
\[
E^2_{*,*} = \HH_*(H_*(R; k)) \Rightarrow H_*(\THH(R); k).
\]
This spectral sequence arises from the skeletal filtration of the simplicial spectrum $\THH(R)_{\bullet}$. Analogously, for a coalgebra spectrum $C$, we consider the Bousfield--Kan spectral sequence arising from the cosimplicial spectrum $\coTHH^\bullet(C)$. We call this the coB\"okstedt spectral sequence. We identify the $E_2$-term in this spectral sequence and see that, as in the $\THH$ case, the $E_2$-term is given by a classical algebraic invariant:

\begin{thm}\label{thm:coBoekstedt} Let $k$ be a field. Let $C$ be a coalgebra spectrum that is cofibrant as an underlying spectrum. The Bousfield--Kan spectral sequence for $\coTHH(C)$ gives a  {\em coB\"okstedt spectral sequence}  with  $E_2$-page
\[ E_2^{s, t}=\coHH_{s,t}^{k}(H_*(C;k)),\]
that abuts to
\[H_{t-s}(\coTHH(C);k).\]
\end{thm}
As one would expect with a Bousfield--Kan spectral sequence, the coB\"okstedt spectral sequence does not always converge. However, we identify conditions under which it converges {completely}. In particular, if the coalgebra spectrum $C$ is a suspension spectrum of a simply connected space $X$, denoted $\Sigma_+^{\infty}X$, the coB\"okstedt spectral sequence converges {completely} to $H_*(\coTHH(\Sigma_+^{\infty}X); k)$ {if a Mittag-Leffler condition is satisfied.}

Further, we prove that for $C$ connected and cocommutative this is a spectral sequence of coalgebras. Using this additional algebraic structure we prove several computational results, including the following. The divided power coalgebra $\Gamma_k[-]$ and the exterior coalgebra $\Lambda_k(-)$ are introduced in detail in Section~\ref{sect:computations}. Note that if $X$ is a graded $k$-vector space concentrated in even degrees and with basis $ ( x_i )_{i\in \mathbb{N}}$, there are identifications 
$\symcofree(X) \cong \Gamma_k[x_1,x_2,\dots]$
and 
$\Omega^{S^c(X) \vert k} \cong U(\Gamma_k[x_1,x_2,\dots]) \otimes U(\Lambda_k[z_1,z_2,\dots])$. We improve Theorem \ref{thm:HKR} in Proposition \ref{cohhcomp} by showing that under the above conditions on $X$ there is indeed an isomorphism of coalgebras
$$\coHH(S^c(X)) \cong \Gamma_k[x_1,x_2,\dots]) \otimes \Lambda_k[z_1,z_2,\dots].$$ Hence Theorem \ref{thm:coBoekstedt} and the coalgebra structure on the spectral sequence yield the following result.

\begin{thrm}\label{comp1} Let $C$ be a cocommutative coassociative coalgebra spectrum that is cofibrant as an underlying spectrum, and whose homology coalgebra is  
\[H_*(C;k) = \Gamma_{k}[x_1,x_2,\dots],\] where the $x_i$ are cogenerators in nonnegative even degrees,
 and there are only finitely many cogenerators in each degree. Then the coB\"okstedt spectral sequence for $C$ collapses at $E_2$, and 
\[ E_2\cong E_{\infty} \cong \Gamma_{k}[x_1, x_2, \dots] \otimes \Lambda_{k}(z_1, z_2, \dots) ,\]
with $x_i$ in degree $(0, \deg(x_i))$ and $z_i$ in degree $(1, \deg(x_i))$. 
\end{thrm}

This theorem should thus be thought of as a computational analogue of the Hochschild--Kostant--Rosenberg Theorem at the level of homology.

These computational results apply to determine the homology of $\coTHH$ of many suspension spectra, which have a coalgebra structure induced by the diagonal map on spaces.  For simply connected $X$,  topological coHochschild homology of  $\Sigma_+^\infty X$ coincides with topological Hochschild homology of  $\Sigma^{\infty}_+\Omega X$, and thus is closely related to algebraic $K$-theory, free loop spaces, and string topology.  Our results give a new way of approaching these invariants, as we briefly recall in Section \ref{sect:computations}.

The paper is organized as follows. In Section \ref{sect:definitions}, we define $\coTHH$ for a coalgebra in any model category endowed with a symmetric monoidal structure as the homotopy limit of the cosimplicial object given by the cyclic cobar construction $\coTHH^{\bullet}$.  We then give conditions for when this homotopy limit can be calculated efficiently and explain a variety of examples.

In Section \ref{sect:categoricalhochschildhomology}, we develop cofree coalgebra functors. These are not simply dual to free algebra functors, and we are explicit about the conditions on a symmetric monoidal category that are necessary to make these functors well behaved. In this context an identification similar to the Hochschild--Kostant--Rosenberg Theorem holds on the cosimplicial level.  We also prove a dg-version of the Hochschild--Kostant--Rosenberg theorem.

In Section 4 we define and study the coB\"okstedt spectral sequence for a coalgebra spectrum. In particular, we analyze the convergence of the spectral sequence, and prove that it is a spectral sequence of coalgebras. 

In Section 5 we use the coB\"okstedt spectral sequence to make computations of the homology of $\coTHH(C)$ for certain coalgebra spectra $C$. We exploit the coalgebra structure in the coB\"okstedt spectral sequence to prove that the spectral sequence collapses at $E_2$ in particular situations and give several specific examples.

\subsection{Acknowledgements}  The authors express their gratitude to the organizers of the Women in Topology II Workshop and the Banff International Research Station, where much of the research presented in this article was carried out.  We also thank Kathryn Hess, Ben Antieau, and Paul Goerss for several helpful conversations. The second author was supported by the National Science Foundation CAREER Grant, DMS-1149408. The third author was supported by the DNRF Niels Bohr Professorship of Lars Hesselholt. The fourth author was supported by the National Science Foundation, DMS-140648. The participation of the first author in the workshop was partially supported by the AWM's ADVANCE grant.

\section{coHochschild homology: Definitions and examples}\label{sect:definitions}

Let $(\D, \otimes, \one)$ be a symmetric monoidal category. We will suppress the associativity and unitality isomorphisms that are part of this structure in our notation. The coHochschild homology of a coalgebra in $\D$ is defined as the homotopy limit of a cosimplicial object in $\D$.  In this section, we define coHochschild homology and then discuss a variety of algebraic examples.
\begin{defn} 
A {\em coalgebra} $(C, \comult)$ in $\D$ consists of an object $C$ in $\D$ together with a morphism
$\comult \colon C \ra C \otimes C,$
called comultiplication, which is coassociative, \ie satisfies
\[(C \otimes \comult) \comult = (\comult \otimes C)\comult.\]
Here we use $C$ to denote the identity morphism on the object $C\in\D$.
A coalgebra is called \emph{counital} if it admits a  {\em counit} morphism $\epsilon \colon C \ra \one$ such that
\[(\epsilon \otimes C) \comult = C = (C \otimes \epsilon) \comult.\]
Finally, a counital coalgebra $C$ is called {\em coaugmented} if there furthermore is a coaugmentation morphism $\eta \colon \one \ra C$
satisfying the identities
\[ \comult \eta = \eta \otimes \eta \quad \text{ and } \quad \epsilon \eta = \one.\]
We denote the category of coaugmented coalgebras by $\coalgD$.

The coalgebra $C$ is called {\em cocommutative} if
\[\tau_{C,C} \comult = \comult,\]
with $\tau_{C,D} \colon C \otimes D \ra D \otimes C$ the symmetry isomorphism. The category of cocommutative coaugmented coalgebras in $\D$ is denoted by $\comcoalgD$.
\end{defn}

We often do not distinguish between a coalgebra and its underlying object. 

Given a counital coalgebra, we can form the associated cosimplicial coHochschild object:

\begin{defn}\label{defn:coTHH}
Let $(C, \comult, \epsilon)$ be a counital coalgebra in $\D$. 
We define a cosimplicial object $\coTHH^{\bullet}(C)$ in $\D$ by setting
\[
\coTHH^{n}(C) = C^{\otimes n+1}
\]
with cofaces
\[\delta_i \colon C^{\otimes n+1} \ra C^{\otimes n+2}, \quad \delta_i=\begin{cases}
C^{\otimes i} \otimes \comult \otimes C^{n-i}, & 0\leq i \leq n,\\
{\tau_{C, C^{\otimes n+1}}}(\comult \otimes  C^{\otimes n}), & i=n+1,
\end{cases}\]
and codegeneracies
\[\sigma_i \colon C^{\otimes n+2} \ra C^{\otimes n+1}, \quad
\sigma_i = C^{\otimes i+1} \otimes \epsilon \otimes C^{\otimes n-i} \quad \text{for } 0 \leq i \leq n.\]
\end{defn}

Note that if $C$ is cocommutative, $\comult$ is a coalgebra morphism and hence $\coTHH^{\bullet}(C)$ is a cosimplicial counital coalgebra in $\D$.

Recall that if $\D$ is a model category, by \cite[16.7.11]{hirschhorn} the category $\D^{\Delta}$ of cosimplicial objects in $\D$ is a Reedy framed diagram category, since $\Delta$ is a Reedy category. In particular, simplicial frames exist in $\D$. Hence we can form the homotopy limit $\holim_{\Delta} X^{\bullet}$ for $X \in \D^{\Delta}$.  More precisely, we use the model given by Hirschhorn \cite[Chapters  16 and 19]{hirschhorn} and in particular apply \cite[19.8.7]{hirschhorn}, so that $\holim_\Delta(X_\bullet)$ can be computed via totalization as
\[\holim_{\Delta} X_{\bullet} = \Tot Y_{\bullet}\]
where $Y_\bullet$ is a Reedy fibrant replacement of $X_\bullet$. Note that in order to apply \cite[19.8.7]{hirschhorn}, it suffices to have a Reedy framed diagram category structure.

\begin{defn}\label{defn:coTHHholim}
Let $\D$ be a model category and a symmetric monoidal category. Let $(C, \comult, 1)$ be a counital coalgebra in $\D$. Then the \emph{coHochschild homology} of $C$ is defined by
\[\coTHH(C) = \holim_{\Delta} \coTHH^{\bullet}(C).\]
\end{defn}

We chose to denote coHochschild homology in a category $\D$ as above by $\coTHH$ instead of $\coHH$ to emphasize that the construction relies on the existence of a model structure on $\D$. However, in several of the examples that follow, the category $\D$ does not come from topological spaces or spectra.

For example we will see in \ref{ex:cochBounded} and \ref{ex:ch} that $\coTHH$ coincides with the original definition of coHochschild homology for $k$-modules and more generally cochain coalgebras (see \cite{doi}, \cite{hps}). We will use the standard notation $\coHH(C)$ for the coHochschild homology of a coalgebra $C$ in the category of graded $k$-modules in later sections.

\begin{rem}\label{rem:fibrancyassumptionsforholim}
Let $\D$ be a symmetric monoidal model category. Let $f \colon C \ra D$ be a morphism of counital coalgebras and assume that $f$ is a weak equivalence and $C$ and $D$ are cofibrant in $\D$. Then $f$ induces a degreewise weak equivalence of cosimplicial objects $\coTHH^{\bullet}(C) \ra \coTHH^{\bullet}(D)$. Hence the corresponding Reedy fibrant replacements are also weakly equivalent and $f$ induces a weak equivalence $\coTHH(C) \ra \coTHH(D)$ according to \cite[19.4.2(2)]{hirschhorn}.
\end{rem}

A central question raised by the definition of coHochschild homology  is that of understanding the homotopy limit over $\Delta$. In the rest of this section, we discuss conditions under which one can efficiently compute it, including an assortment of algebraic examples which we discuss further in Section \ref{sect:categoricalhochschildhomology}.

For a cosimpicial object $X^{\bullet}$ in $\D$ let $M_n(X^{\bullet})$ be its $n$th matching object given by
\[ 
M_n(X^{\bullet}) = \lim_{\substack{\alpha \in \Delta([n],[a]), \\\alpha \text{ surjective}, \\a \neq n}} X^a.
\]
The matching object $M_n(X^{\bullet})$ is traditionally also denoted by $M^{n+1}(X^{\bullet})$, but we stick to the conventions in \cite[15.2]{hirschhorn}.
If $\D$ is a model category, $X^{\bullet}$ is called \emph{Reedy fibrant} if for every $n$, the morphism
\[\sigma \colon X^n \ra M_n(X^{\bullet})\]
is a fibration, where $\sigma$ is induced by the collection of maps $\alpha_* \colon X^n \ra X^a$ for $\alpha\in \Delta([n],[a])$.

If $\D$ is a simplicial model category and $X^{\bullet}$ is Reedy fibrant, $\holim_{\Delta} X^{\bullet}$  is weakly equivalent to the totalization of $X^{\bullet}$ defined via the simplicial structure; see \cite[Theorem 18.7.4]{hirschhorn}. In this case the totalization is given by the equalizer
\[\Tot (X^{\bullet}) = \text{eq}\Bigg(
\prod_{n \geq 0} (X^n)^{\Delta^n} \rightrightarrows \prod_{\alpha \in \Delta([a],[b])} (X^b)^{\Delta^a}
\Bigg),\]
where $D^K$ denotes the cotensor of an object $D$ of $\D$ and a simplicial set $K$ and where the two maps are induced by the cosimplicial structures of $X^{\bullet}$ and  $\Delta^{\bullet}$.  
If the category $\D$ is not simplicial, one has to work with simplicial frames as in \cite[Chapter 19]{hirschhorn}.

In the dual case, if $A$ is a monoid in a symmetric monoidal model category, one considers the cyclic bar construction associated to $A$, \ie the simplicial object whose definition is entirely dual to the definition of $\coTHH^{\bullet}(C)$.  This simplicial object is Reedy cofibrant under mild conditions; see for example \cite[Example 23.8]{shulman}. But the arguments proving Reedy cofibrancy of the simplicial Hochschild object do not dualize to $\coTHH^{\bullet}(C)$: It is neither reasonable to expect the monoidal product to commute with pullbacks, nor to expect that pullbacks and the monoidal product satisfy an analogue of the pushout-product axiom. 

Nonetheless, in certain cases $\coTHH^{\bullet}(C)$ can be seen to be Reedy fibrant. In particular Reedy fibrancy is simple when the symmetric monoidal structure is given by the cartesian product.

\begin{lemma} \label{lem:ReedyFibTimes}
Let $\D$ be any category and let $C$ be a coalgebra in $(\D, \times, *)$. Then
\[\sigma \colon \coTHH^n(C) \ra M_n(\coTHH^{\bullet}(C))\]
is an isomorphism for $n\geq 2$.  Hence if $\D$ is a model category, then $\coTHH^\bullet(C)$ is Reedy fibrant if $C$ itself is fibrant.
\end{lemma}
\begin{proof}
For surjective $\alpha \in \Delta([n],[n-1])$, let 
\[p_{\alpha} \colon M_n(\coTHH^{\bullet}(C)) \ra \coTHH^{n-1}(C)=C^{\times n}\]
 be the canonical projection to the copy of $\coTHH^{n-1}(C)$ corresponding to $\alpha$. Then an inverse to $\sigma$ is given by
\begin{multline*}
M_n(\coTHH^{\bullet}(C)) \xto{(p_{\alpha})_{\alpha}}  (C^{\times n})^{\times n} \xto{((C^{\times n})^{\times n},C^{\times n} \times (*)^{\times n-1})} (C^{\times n})^{n+1}\\
  \xto{(C\times *^{n-1}) \times \dotsb \times (*^{n-1} \times C) \times (*^{n-1} \times C)} C^{\times n+1}. 
\end{multline*}
Intuitively, this map sends $(c^1,\dotsc,c^n) \in M_n(\coTHH^{\bullet}(C))$ with $c^i=(c^i_1,\dotsc,c^i_n) \in C^{\times n}$ to
$(c^1_1,\dotsc, c^n_n, c^1_n).$
The Reedy fibration conditions for $n = 1, 0$ follow from the fibrancy of $C$.
\end{proof}

In several algebraic examples, degreewise surjections are fibrations and the following lemma will imply Reedy fibrancy.
\begin{lemma} \label{lem:ReedyFibSurj}
Let $C$ be a counital coalgebra in the category of graded $k$-modules for a commutative ring $k$.  If $C$ is coaugmented, or if there is a $k$-module map $\eta\colon k\to C$ such that $\epsilon\eta=C$, then the matching maps
\[ \sigma\colon \coTHH^n(C) \to M_n(\coTHH^\bullet(C)) \]
are surjective.
\end{lemma}
In the rather technical proof we construct an explicit preimage for any element in $M_n(\coTHH^\bullet(C))$. We defer the proof to Appendix \ref{appendix:yuckyproof}. Note that if $k$ is a field, then by choosing a basis for $C$ over $k$ compatibly with the counit, we can always construct the necessary map $\eta$.  Thus $\coTHH$ of any counital coalgebra over a field is Reedy fibrant.

We now collect a couple of examples of what $\coTHH$ is in various categories. 

\begin{ex}\label{ex:coskmod}
Let $\D=(\ckmod, \otimes, k)$ be the symmetric monoidal category of cosimplicial $k$-modules for a field $k$. There is a simplicial model structure on $\D$; see \cite[II.5]{fresse} for an exposition. The simplicial cotensor of a cosimplicial $k$-module $M^{\bullet}$ and a simplicial set $K_{\bullet}$ is given by
\[(M^K)^n= \Set(K_n,M^n)\]
with $\alpha \in \Delta([n],[m])$ mapping $f\colon K_n \ra M^n$ to
\[K_m \xto{\,\alpha\,}  K_n \xto{\,f\,}  M_n \xto{\,\alpha_*\,}  M_m,\]
see \cite[II.2.8.4]{goerss-jardine}.
As in the simplicial case (see \cite[III.2.11]{goerss-jardine}), one can check that the fibrations are precisely the degreewise surjective maps.
Hence by Lemma \ref{lem:ReedyFibSurj} $\coTHH(C)$ can be computed as the totalization, that is, as the diagonal of the bicosimplicial $k$-module $\coTHH^{\bullet}(C)$. 
Even though $\D$ is not a monoidal model category, the tensor product of weak equivalences is again a weak equivalence, and hence a weak equivalence of counital coalgebras  induces  a weak equivalence on $\coTHH$.
\end{ex}

Recall that the dual Dold--Kan correspondence is an equivalence of categories
\[
N^* \colon \ckmod \leftrightarrows\dgm^{\geq0}(k) \ \colon\! \Gamma^{\bullet}
\]
between cosimplicial $k$-modules and nonnegatively graded cochain complexes over $k$. Given $M^{\bullet} \in \ckmod$, the normalized cochain complex $N^*(M^{\bullet})$ is defined by
\[N^a(M^{\bullet})=\bigcap_{i=0}^{a-1} \ker (\sigma_i \colon M^a \ra M^{a-1})\]
for $a>0$ and $N^0(M^{\bullet}) = M^0$, with differential $d \colon N^a(M^{\bullet}) \ra N^{a+1}(M^{\bullet})$ given by
\[\sum_{i=0}^{a+1} (-1)^i \delta_i.\]

\begin{ex}\label{ex:cochBounded}

Let $k$ be a field. Let $\D$ be the  symmetric monoidal category $\D=(\dgm^{\geq0}(k), \otimes, k)$ of nonnegatively graded cochain complexes over $k$. The dual Dold--Kan correspondence
yields that $\D$ is a simplicial model category, with weak equivalences the quasi-isomorphisms and with fibrations the degreewise surjections.
Hence by Lemma \ref{lem:ReedyFibSurj} we have that $\coTHH^{\bullet}(C)$ is Reedy fibrant for any counital coalgebra $C$ in $\D$.
The totalization of a cosimplicial cochain complex $M^{\bullet}$ is given by
\[\Tot(M^{\bullet}) \cong N^*(\Tot(\Gamma(M^{\bullet}))) = N^*(\diag(\Gamma(M^{\bullet}))).\]

Recall that given a double cochain complex $B_{*,*}$ with differentials $d^h \colon B_{*,*} \ra B_{*+1,*}$ and $d^v\colon B_{*,*} \ra B_{*,*+1}$, we can form the total cochain complex  $\Tot_{\dgm}(B)$ given by
\[\Tot_{\dgm}(B)_n = \prod_{p+q =n} B_{p,q}.\]
The projection to the component $B_{s, n+1-s}$ of the differential applied to $(b_{p} \in B_{p,n-p})_p$  is given by
$d^h(b_{s-1,n-s+1}) +(-1)^pd^v(b_{s,n-s}).$ A similar construction can be carried out for double chain complexes.

As in the bisimplicial case, we can turn a bicosimpicial $k$-module $X^{\bullet,\bullet}$ into a double complex by applying the normalized cochain functor to both cosimplicial directions. The homotopy groups of $\diag X^{\bullet,\bullet}$ are the homology of the associated total complex, hence 
\[\coTHH(C) = \Tot_{\dgm} N^*(\coTHH^{\bullet}(C))\]
for any counital coalgebra in $\D$. Hence our notion of coHochschild homology coincides with the classical one as in \cite{doi} and \cite{hps}. Again weak equivalences of counital coalgebras induce weak equivalences on $\coTHH$.
\end{ex}

\begin{ex}\label{ex:simpset}
The category $\D = (\sSet, \times, *)$ is a simplicial symmetric monoidal model category with its classical model structure. Every object is cofibrant, and
by Lemma \ref{lem:ReedyFibTimes}, $\coTHH(C)$ can be computed as the totalization of the cosimplicial simplicial set $\coTHH^{\bullet}(C)$ when $C$ is fibrant.
\end{ex}

\begin{ex}\label{ex:skmod}
Let $\D = (\skmod, \otimes, k)$ be the category of simplicial $k$-modules over a field $k$. Every object is cofibrant in $\D$. 
 Let $C$ be a counital coalgebra. Since every morphism that is degreewise surjective is a fibration (see  \cite[III.2.10]{goerss-jardine}), Lemma \ref{lem:ReedyFibSurj} yields that $\coTHH^{\bullet}(C)$ is Reedy fibrant, and $\coTHH(C)$ is given by the usual totalization of cosimplicial simplicial $k$-modules.
\end{ex}

\begin{ex} \label{ex:chBounded} 
The category $\D=(\dgm_{\geq0}(k), \otimes, k)$ of nonnegatively graded chain complexes over a ring $k$ with the projective model structure is both a monoidal model category as well as a simplicial model category via the Dold--Kan equivalence $(N, \Gamma)$ with $(\skmod, \otimes, k)$.
If we work over a field, every object is cofibrant, and for any counital coalgebra $C$ in $\D$ Lemma \ref{lem:ReedyFibSurj} yields that $\coTHH^{\bullet}(C)$ is Reedy fibrant.

However, a word of warning is in order. The definition of coHochschild homology given in Definition \ref{defn:coTHHholim} does not coincide with the usual notion of coHochschild homology of a differential graded coalgebra as for example found in \cite{doi}:  Usually, coHochschild homology of a coalgebra $C$ in $k$-modules or differential graded $k$-modules is given by applying the normalized cochain complex functor $N^*$ to $\coTHH^{\bullet}(C)$ and forming the total complex $\Tot_{\dgm}(N^*(\coTHH^{\bullet}(C)))$ of the double chain complex resulting from interpreting $N^a(\coTHH^{\bullet}(C))$ as living in chain complex degree $-a$. 

But for any cosimplicial nonnegatively graded chain complex $M^{\bullet}$, 
\[\Tot(M^{\bullet}) \cong N_*(\Tot(\Gamma(M^{\bullet}))),\]
where the second $\Tot$ is the totalization of the simplicial cosimplicial $k$-module $\Gamma(M^{\bullet})$.
As explained in \cite[III.1.1.13]{fresse}, we have an identity
\[N_*(\Tot(\Gamma(M^{\bullet}))) = \tau_*(\Tot_{\dgm}(N^*(M^{\bullet}))),\]
where $\tau_*(C)$ of an unbounded chain complex $C$ is given by
\[\tau_*(C)_n = \begin{cases}
C_n, & n>0,\\
\ker(C_0 \ra C_{-1}), & n=0,\\
0, n<0.
\end{cases}\]

Hence if $C$ is concentrated in degree zero, so is $\coTHH(C)$, in contrast to the usual definition as for example found in \cite{doi} and \cite{hps}.

Nonetheless, if we assume that $C_0 =0$, the above identification of the totalization yields that the two notions of coHochschild homology coincide.
\end{ex}

\begin{ex}\label{ex:ch}
The category $\D=(\dgm, \otimes, k)$ of unbounded chain complexes over a commutative ring $k$ is a symmetric monoidal model category with weak equivalences  the quasi-isomorphisms and fibrations the degreewise surjective morphisms; see \cite[4.2.13]{hovey}. 

Let $C$ be a cofibrant chain complex. A simplicial frame of $C$ is given by the simplicial chain complex $\fr(C)$ which is given in simplicial degree $m$ by the internal hom object $\underline{\dgm}(N_*(\Delta^m), C)$ (see proof of \cite[5.6.10]{hovey}). 
Again, $\coTHH^{\bullet}(C)$ is Reedy fibrant. This yields that the homotopy limit $\holim_{\Delta} \coTHH^{\bullet}(C)$ coincides
with the total complex of the double complex $N^*(\coTHH^{\bullet}(C))$. Again this coincides with the definitions of coHochschild homology in \cite{doi} and \cite{hps}.
\end{ex}

\begin{ex}
Both the category of symmetric spectra as well as the category of $\mathbb{S}$-modules are simplicial symmetric monoidal model categories. Sections \ref{sect:thespectralsequence} and \ref{sect:computations} discuss $\coTHH$ for coalgebra spectra.
\end{ex}

\section{coHochschild homology of cofree coalgebras in symmetric monoidal categories}\label{sect:categoricalhochschildhomology}

We next turn to calculations of coHochschild homology of cofree coalgebras in symmetric monoidal categories.  The first step is to specify what we mean by a ``cofree'' coalgebra in our symmetric monoidal category  $\D$.  Intuitively, such a coalgebra should be given by a right adjoint to the forgetful functor from the category of coalgebras in $\D$ to $\D$ itself. The constructions are the  usual ones and can be found for example in \cite[II.3.7]{markl-shnider-stasheff} in the algebraic case, but we specify conditions under which we can guarantee the existence of this adjoint. For an operadic background to these conditions we refer the reader to \cite{ching}.

Let $\D$ be complete and cocomplete and admit a zero object $0$. Given a coalgebra $(C, \comult)$ in a symmetric monoidal category $\D$ and $n\geq 0$, we denote by $\comult^n \colon C \ra C^{\otimes n+1}$ the iterated comultiplication defined inductively by
\[\comult^{0} = C, \qquad \comult^{n+1} = (\comult \otimes C^{\otimes n}) \comult^n.\]

Note that if $(C, \comult, \epsilon, \eta)$ is a coaugmented coalgebra, the cokernel of the coaugmentation, $\coker \eta$, is a coalgebra.

\begin{defn}
 We call a (non-coaugmented) coalgebra $C$ {\em conilpotent} if the morphism
\[  (\comult^n)_{n\geq 0} \colon C \ra \prod_{n\geq {0}} C^{\otimes n+1},\]
that is, the morphism given by $\comult^n$ after projection to $C^{\otimes n+1}$, factors through
$\bigoplus_{n \geq {0}} C^{\otimes n+1}$
via the canonical morphism 
\[\bigoplus_{n \geq {0}} C^{\otimes n+1} \ra \prod_{n \geq {0}} C^{\otimes n+1}.\]
We say that a  coaugmented coalgebra $C$ is conilpotent if $\coker \eta$ is a conilpotent coalgebra.

We denote the full subcategory of $\coalgD$ consisting of coaugmented conilpotent coalgebras by $\coalgDnil$. Similarly, the category of cocommutative conilpotent coaugmented coalgebras is denoted by $\comcoalgDnil$.
\end{defn}

\begin{defn}\label{defn:cofreefriendly}
The symmetric monoidal category $(\D, \otimes, \one)$ is called \emph{cofree-friendly} 
if it is complete, cocomplete, admits a zero object, and the following additional conditions hold:
\begin{enumerate}
\item For all objects $D$ in $\D$, the functor $D \otimes -$ preserves colimits.
\item Finite sums and finite products are naturally isomorphic, that is for every finite set $J$ and all objects $D_j$, $j \in J$, the morphism 
\[\bigoplus_{j \in J} D_j \ra \prod_{j \in J} D_j\]
induced by the identity morphism on each $D_j$ is an isomorphism.
\end{enumerate}
\end{defn} 

\begin{prop}
If $\D$ is cofree-friendly, the functor 
\[coker \colon \coalgDnil \ra \D, \quad (C,\comult, \epsilon, \eta) \mapsto \coker \eta\]
admits a right adjoint $\cofree$,  which we call the \emph{cofree coaugmented conilpotent coalgebra functor.}
The underlying object of $\cofree (X)$ is given by $\bigoplus_{n \geq 0} X^{\otimes n}$. The comultiplication $\comult$ is defined via deconcatenation, \ie by 
\addtolength{\jot}{1ex}
\begin{multline*}
\bigoplus_{n\geq 0} X^{\otimes n}\xto{\bigoplus_{n\geq 0} (X^{\otimes n})_{a+b=n}} 
 \bigoplus_{n \geq 0} \prod_{\substack{ a,b\geq 0 \\ a+b=n}} X^{\otimes n} \xto{\cong} 
 \bigoplus_{n \geq 0} \bigoplus\limits_{\substack{ a,b\geq 0 \\ a+b=n}} X^{\otimes n} \\
\cong \bigoplus\limits_{a,b\geq 0} X^{\otimes a} \otimes X^{\otimes b}
\xto{\cong} 
  \bigoplus\limits_{a\geq 0} X^{\otimes a} \otimes \bigoplus\limits_{b\geq 0} X^{\otimes b},
\addtolength{\jot}{-1ex}
\end{multline*}
where the isomorphism on the first line is provided by condition (2) of Definition \ref{defn:cofreefriendly}.
The counit morphism $\pi \colon \cofree(X) \ra X$ is the identity on $X$ and the zero morphism on the other summands.
\end{prop}

\begin{ex}\label{ex:kmodcofreefriendly}
The categories of chain or cochain complexes discussed in Examples \ref{ex:cochBounded}, \ref{ex:chBounded} and \ref{ex:ch} are cofree-friendly, and the cofree coalgebra $T^c(X)$ generated by an object in any of these categories is the usual tensor coalgebra.

Similarly, the categories of simplicial or cosimplicial $k$-modules discussed in Examples \ref{ex:coskmod} and \ref{ex:skmod} are cofree-friendly, and the cofree coalgebra generated by $X$ is obtained by applying the tensor coalgebra functor in each simplicial or cosimplicial degree.
\end{ex}
\begin{ex}\label{ex:notcofreefriendlythings}
The categories of simplicial sets, of pointed simplicial sets and of spectra are not cofree-friendly, since finite coproducts and finite products are not isomorphic in any of these categories.
A more abstract way of thinking about this is given by observing that in these categories coaugmented coalgebras can not be described as coalgebras over a cooperad in the usual way: While the category of symmetric sequences in these categories is monoidal with respect to the plethysm that gives rise to the notion of operads, the plethysm governing cooperad structures does not define a monoidal product. This is discussed in detail by Ching \cite{ching}; see in particular Remark 2.10, Remark 2.20 and Remark 2.21.
\end{ex}

Let $G$ be a group with identity element $e$. Recall that a $G$-action on an object $X$ in $\D$ consists of morphisms
\[\phi_g \colon X \ra X, \qquad g \in G,\]
such that 
$\phi_g \phi_h =  \phi_{gh}$ and $\phi_e = X$.
The {\em fixed points} $X^G$ of this actions are given by the equalizer
\[X^G = \eq(X \xrightrightarrows[(\phi_g)_{g \in G}]{(X)_{g\in G}} \prod_{g\in G} X).\]

\begin{defn}\label{defn:permutationfriendly}
Recall that for any object $X$ of $\D$ the symmetric group $\Sigma_n$ acts on $X^{\otimes n}$ by permuting the factors. We call $\D$ {\em permutation-friendly}
if the morphism 
\[ (X^{\otimes a})^{\Sigma_a} \otimes (X^{\otimes b})^{\Sigma_b} \ra (X^{\otimes a+b})^{\Sigma_a \times \Sigma_b},\]
induced by the morphisms $ (X^{\otimes a})^{\Sigma_a} \ra X^{\otimes a} $ and $ (X^{\otimes b})^{\Sigma_b} \ra X^{\otimes b}$, 
is an isomorphism for all $X$ in $\D$ and all $a,b\geq 0$.
\end{defn}

\begin{prop}
If $\D$ is cofree-friendly and permutation-friendly, the forgetful functor 
\[\coker\!\!\colon \comcoalgDnil \ra \D, \quad (C,\comult,\epsilon,\eta) \mapsto \coker \eta\]
admits a right adjoint $\symcofree$, which is called the \emph{cofree cocommutative conilpotent coaugmented coalgebra functor}.
For a given object $X$ of $\D$, define $\symcofree (X)=\bigoplus_{n \geq 0} (X^{\otimes n})^{\Sigma_n}$. The comultiplication $\comult$ on $\symcofree(X)$ is defined by
\addtolength{\jot}{1ex}
\begin{equation*}
\begin{split}
\bigoplus_{n\geq 0} (X^{\otimes n})^{\Sigma_n}
&\xto{\bigoplus_{n\geq 0} ((X^{\otimes n})^{\Sigma_n})_{a+b=n}}
\bigoplus\limits_{n \geq 0} \prod\limits_{\substack{ a,b \geq 0, \\a+b=n}} (X^{\otimes n})^{\Sigma_n} 
\xto{\cong} 
\bigoplus_{n \geq 0} \bigoplus_{\substack{ a,b \geq 0, \\a+b=n}} (X^{\otimes n})^{\Sigma_n}\\ 
&\cong \bigoplus\limits_{a,b\geq 0} (X^{\otimes a} \otimes X^{\otimes b})^{\Sigma_{a+b}}
\xto{\bigoplus\limits_{a,b\geq 0}\res_{a,b}}
 \bigoplus\limits_{a,b\geq 0} (X^{\otimes a} \otimes  X^{\otimes b})^{\Sigma_a \times \Sigma_b}\\
&\xto{\cong}
\bigoplus\limits_{a,b\geq 0} (X^{\otimes a})^{\Sigma_a} \otimes  (X^{\otimes b})^{\Sigma_b}
\xto{\cong}
 \bigoplus\limits_{a\geq 0} (X^{\otimes a})^{\Sigma_a} \otimes \bigoplus\limits_{b\geq 0} (X^{\otimes b})^{\Sigma_b}.
\end{split}
\addtolength{\jot}{-1ex}
\end{equation*}
The isomorphism on the first line again exists because of condition (2) of Definition \ref{defn:cofreefriendly} and the isomorphism at the beginning of the last line is derived from  Definition \ref{defn:permutationfriendly}.  The map $\res_{a,b}$ is the map  
\[\res_{a,b} \colon  (X^{\otimes a} \otimes X^{\otimes b})^{\Sigma_{a+b}} \ra (X^{\otimes a} \otimes  X^{\otimes b})^{\Sigma_a \times \Sigma_b}\]
induced by the inclusion $\Sigma_a \times \Sigma_b \ra \Sigma_{a+b}$.
The counit morphism $\pi\colon\symcofree(X) \ra X$ is given by the identity on $X$ and by the zero morphism on the other summands.
\end{prop}

\begin{ex}\label{ex:permutfriendlythings}
If $k$ is a field, 
all the categories  discussed in Example~\ref{ex:kmodcofreefriendly} are permutation friendly in addition to being cofree friendly. This means that the categories of nonnegatively graded (co)chain complexes and unbounded chain complexes of $k$-modules and the categories of simplicial and cosimplicial $k$-modules all admit well-defined cofree cocommutative coalgebra functors.

The permutation friendliness of each of these categories can be deduced from the permutation friendliness of the category of graded $k$-modules. A proof for graded $k$-modules concentrated in even degrees or for $k$ a field of characteristic $2$ can be found in Bourbaki \cite[p. IV.49]{bourbaki}. If the characteristic of $k$ is different from $2$ and $M$ is concentrated in odd degrees, $(M^{\otimes n})^{\Sigma_n}$ and $(M^{\otimes a} \otimes M^{\otimes b})^{\Sigma_a \times \Sigma_b}$ can be described explicitly in terms of a basis of $M$.
For an arbitrary graded module $M$, the claim follows from writing $M$ as the direct sum of its even and its odd degree part.

More concretely, if $M$ admits a countable basis $m_i$, $i \geq 1,$ and the $m_i$ are even degree elements, or if the characteristic of $k$ is $2$, we identify $S^c(M)$ with the Hopf algebra $\Gamma_k[m_1, m_2,\dotsc]$.  As an algebra, $\Gamma_k[m_1, m_2,\dotsc]$ is the divided power algebra generated by $M$. The isomorphism to $S^c(M)$ is given by identifying the element $\gamma_{j_1}(m_{i_1})\dotsm\gamma_{j_n}( m_{i_n}) \in \Gamma_k[m_1, m_2,\dotsc]$ with the element $\sum_{\sigma \in \sh(j_1,\dotsc,j_n)} \sigma.(m_{i_1}^{\otimes j_1} \otimes \dotsb \otimes m_{i_n}^{j_n}) \in S^c(M),$
where $\sigma$ ranges over the set of $(j_1,\dotsc,j_n)$-shuffles.
The coproduct is given on multiplicative generators by 
\[\comult(\gamma_j(m_i)) = \sum_{a+b=j} \gamma_a(m_i) \otimes \gamma_b(m_i).\]

If $M$ is concentrated in odd degrees and the characteristic of $k$ is different from $2$, we can identify $S^c(M)$ with the Hopf algebra $\Lambda_k(m_1, m_2,\dotsc).$ This is the exterior algebra generated by $M$, and we identify $m_{i_1}\dotsm m_{i_n}$ with $\sum_{\sigma \in \Sigma_n} \sgn(\sigma) \sigma.(m_{1} \otimes \dotsb \otimes m_{n})$. The coproduct is given by
\[\comult(m_i) = 1 \otimes m_i + m_i \otimes 1.\]

If $k$ is not a field, the cofree cocommutative conilpotent coalgebra cogenerated by a $k$-module $M$ still exists: It is the largest cocommutative subcoalgebra of $\cofree(M)$.
\end{ex}

Now that we have established what we mean by cofree coalgebras in $\D$, we turn to an analysis of 
coHochschild homology for these coalgebras.  These computations should be thought of as a simple case of a dual Hochschild--Kostant--Rosenberg theorem for coalgebras.

Let $(C,\comult_C, \epsilon_C, \eta_C )$ and $(D,\comult_D, \epsilon_D, \eta_D)$ be cocommutative coaugmented coalgebras. Recall that their product in $\comcoalgD$ exists and is given by $C \otimes D$ with comultiplication $ (C\otimes  {\tau_{C,D}} \otimes D) (\comult_C \otimes \comult_D)$, counit $\epsilon_C \otimes \epsilon_D$ and coaugmentation $\eta_C \otimes \eta_D.$ Here $\tau_{C,D} \colon C \otimes D \ra D \otimes C$ is the symmetry isomorphism switching $C$ and $D$.
The projections $C\otimes D \ra C$ and $C\otimes D \ra D$ are given by $C \otimes \epsilon_D$ and $\epsilon_D \otimes C$.

Since $\symcofree$ is a right adjoint, we have the following property:

\begin{lemma} \label{lem:symcofreeAndProducts}
 If $\D$ is cofree-friendly, permutation-friendly and cocomplete, there is a natural isomorphism 
\[\symcofree(X \times Y) \to \symcofree(X) \otimes \symcofree(Y)\]
in $\comcoalgD$. 
It is given by the morphisms 
$\symcofree(X \times Y) \ra \symcofree(X)$ and $\symcofree(X \times Y) \ra \symcofree(Y)$ which are induced by the projections $X \times Y \ra X$ and $X \times Y \ra Y$.
\end{lemma}

Every category $\mathcal{C}$ that admits finite products gives rise to the symmetric monoidal category $(\mathcal{C}, \times, *)$, where $*$ is the terminal object. Every object $X$ in $\mathcal{C}$ is then a counital cocommutative 
coalgebra with respect to this monoidal structure: the comultiplication $X \ra X \times X$ is the diagonal, the counit $X \ra *$ is the map to the terminal object. We indicate the monoidal structure we use to form $\coTHH$ with a subscript, so that if $C$ is a coalgebra with respect to $\otimes$, we write $\coTHH_{\otimes}^{\bullet}(C)$.

\begin{prop}\label{prop:abstractHKR}
Let $\D$ be cofree-friendly and permutation-friendly.
Then for any object $X$ in $\D$
\[\coTHH_{\otimes}^{\bullet}(\symcofree(X)) \cong \symcofree(\coTHH_{\times}^{\bullet}(X)).\]
\end{prop}
\begin{proof}
Applying Lemma \ref{lem:symcofreeAndProducts} yields an isomorphism in each  cosimplicial degree $n$:
\[\coTHH^n_{\otimes}(\symcofree(X)) =\symcofree(X)^{\otimes n+1} \cong \symcofree(X^{\times n+1}) = \symcofree(\coTHH_{\times}^n(X)).\]
To check that the cosimplicial structures agree, it suffices to show that the comultiplications and counits on both sides agree. The diagonal $X \ra X \times X$ induces a morphism of coalgebras
\[\symcofree(X) \ra \symcofree(X \times X).\]
Since $\symcofree(X)$ is cocommutative, the comultiplication 
\[\comult_{\symcofree(X)} \colon \symcofree(X) \ra \symcofree(X) \otimes \symcofree(X) \cong \symcofree(X \times X)\]
is a morphism of coalgebras as well. These maps agree after projection to $X \times X$. 
A similar argument yields that $X \ra 0$ induces the counit of $\symcofree(X)$.
\end{proof}

The following theorem is a consequence of this proposition. This is the best result we obtain about the coHochschild homology of cofree coalgebras without making further assumptions on our symmetric monoidal category $\D$.
\begin{thrm}\label{thrm:coTHHcofreeinsymmonoidalmodelcat}
If $\D$ is a cofree- and permutation-friendly model category, Proposition \ref{prop:abstractHKR} implies that
\[\coTHH_{\otimes}(S^c(X)) \sim \holim_{\Delta}(S^c(\coTHH^{\bullet}_{\times}(X))).\]
\end{thrm}

If $\D$ is a simplicial model category, there is an easy description of $\Tot(\coTHH^{\bullet}_{\times}(X))$ as a free loop object.
\begin{prop}[see \eg \cite{loday-loops}]\label{prop:coTHHcotensorS1simpmodelcat}%
Let $\D$ be a simplicial model category and $X$ a fibrant object of $\D$.
Then 
\[\holim_{\Delta}(\coTHH^{\bullet}_{\times}(X)) \sim  X^{\sphere^1}.\]
\end{prop}
{This result identifies  $\coTHH_\times(X)$ as a sort of ``free loop space'' on $X$.  This is precisely the case in the category of spaces, as in Example \ref{ex:freeloopspace}.}

\begin{ex}
For the category of cosimplicial $k$-modules we can choose the diagonal as a model for the homotopy limit over $\Delta$. Since the cofree coalgebra cogenerated by a cosimplicial $k$-module is given by applying $\symcofree$ degreewise, we actually obtain an isomorphism
of coalgebras
\[\coTHH_{\otimes}(S^c(X))\cong S^c(X^{\sphere^1}).\]
\end{ex}

In the case where $S^c$ is a right Quillen functor, we obtain the following corollary to Theorem \ref{thrm:coTHHcofreeinsymmonoidalmodelcat} and Proposition \ref{prop:coTHHcotensorS1simpmodelcat}.
\begin{cor}
Let $\D$ be a simplicial and a monoidal model category. Let $X$ in $\D$ be fibrant. If there is a model structure on $\comcoalgD$ such that $\symcofree$ is a right Quillen functor, then 
\[\coTHH_{\otimes}(\symcofree(X))\cong \mathbb{R}\symcofree( X^{\sphere^1}).\]
\end{cor}

In general, however,  $S^c$ is not a right Quillen functor.  Up to a shift in degrees which is dual to the degree shift in the free commutative algebra functor, the most prominent counterexample is the category of chain complexes over a ring $k$ of positive characteristic: The chain complex $D^2$ consists of one copy of $k$ in degrees $1$ and $2$, with differential the identity. This complex is fibrant and acyclic, but $S^c(D^2)$ is not acyclic. This same counterexample shows $\symcofree$ is not right Quillen either for unbounded chain complexes or nonnegatively graded cochain complexes.

Nevertheless, a Hochschild--Kostant--Rosenberg type theorem for cofree coalgebras still holds for  coHochschild homology in the category of nonnegatively graded cochain complexes $\dgm^{\geq 0}$ over a field. 
 We denote the underlying differential graded $k$-module of $\symcofree(X)$ by $U(\symcofree(X))$. The desuspension $\Sigma^{-1}X$ of a cochain complex $X$ is the cochain complex with $(\Sigma^{-1}X)^n = X^{n-1}$.
Define the graded $k$-module $\Omega^{\symcofree(X)\vert k}$ by
\[\Omega^{\symcofree(X)\vert k} = \begin{cases}
U(S^c(X)) \otimes U(S^c(\Sigma^{-1}X)), & \text{if } \cha(k) \neq 2,\\
U(S^c(X)) \otimes \Lambda(\Sigma^{-1}X),& \text{if } \cha(k)=2,
\end{cases}
\]
with $\Lambda(M)$ denoting the exterior powers of the cochain complex $M$. We will compare $\Omega^{\symcofree(X) \vert k}$ with the notion of K\"ahler codifferentials as defined in \cite{farinati-solotar-ext} in Remark \ref{rem:KaehlerFS}.

\begin{thrm}\label{thrm:HKRcofreedg}
Let $k$ be a field. Then for $X$ in $\dgm^{\geq 0}$ there is a quasi-isomorphism 
\[\coTHH_\otimes(\symcofree(X))\to \Omega^{\symcofree(X)\vert k}.\]
\end{thrm}

Note that this is an identification of differential graded $k$-modules, not of coalgebras. We determine the corresponding coalgebra structure on $\coTHH_{\otimes}(S^c(X))$ in certain cases in Proposition \ref{cohhcomp}.

This result corresponds to the result of the Hochschild--Kostant--Rosenberg theorem applied to a free symmetric algebra $S(X)$ generated by a chain complex $X$. See, for example, \cite[Theorems 3.2.2 and 5.4.6]{loday}. For a discussion of a coalgebra analogue of  K\"ahler differentials we refer the reader to \cite{farinati-solotar-ext}.

The proof is dual to the proof of the corresponding results for Hochschild homology. We follow the line of proof given by Loday \cite[Theorem 3.2.2]{loday}. We begin by proving a couple of lemmas.
First we identify $\coTHH_{\times}(X)$: 
\begin{lemma}
For a cochain complex $X$ over a field $k$,
\[\Tot_{\dgm} (N^* (\coTHH^{\bullet}_{\times}(X) )) \cong X \times \Sigma^{-1}X.\]
\end{lemma}
\begin{proof}
This follows easily from the fact that 
$$N^0(\coTHH^{\bullet}_{\times}(X)) = X, \qquad N^1(\coTHH^{\bullet}_{\times}(X)) =0 \times X$$
and $N^a(\coTHH^{\bullet}_{\times}(X)) =0$ for $a\geq 2$.
\end{proof}

We next prove the special case of Theorem \ref{thrm:HKRcofreedg} where the module of cogenerators is one dimensional.

\begin{lemma}\label{lem:HKROneGenerator}
Let $k$ be a field. Let $X$ in $\dgm^{\geq 0}$ be concentrated in a single nonnegative degree and be one dimensional in this degree. Then there is a quasi-isomorphism 
\[\coTHH_{\otimes}(S^c(X)) \ra   \Omega^{\symcofree(X)\vert k}.\]
\end{lemma}
\begin{proof}
First let the characteristic of $k$ be $2$ or $X$ be concentrated in even degree, so that $S^c(X) \cong \Gamma_k[x]$ for a generator $x$. Up to the internal degree induced by $X$, we can use the results of Doi \cite[3.1]{doi} to compute $\coTHH_{\otimes}(S^c(X))$ using a $S^c(X) \otimes S^c(X)$-cofree resolution of $S^c(X)$. Such a resolution is given by
$$\xymatrix{S^c(X) \ar[r]^-{\comult} & S^c(X) \otimes S^c(X) \ar[r]^-{f} & S^c(X) \otimes X \otimes S^c(X) \ar[r] & 0,}$$
with $f(x^i \otimes x^j) = x^{i-1} \otimes x \otimes x^j - x^i \otimes x \otimes x^{j-1}$, where $x^{-1}=0$. Hence the homology of $\coTHH_{\otimes}(S^c(X))$ is concentrated in cosimplicial degrees $0$ and $1$. An explicit calculation gives that all elements in $N^0(\coTHH_{\otimes}^{\bullet}(S^c(X))) \cong S^c(X)$ are cycles, while generating cycles in  $N^1(\coTHH_{\otimes}^{\bullet}(S^c(X))) $ 
are given by $\sum_{i=1}^{n} i \cdot x^{n-i} \otimes x^i$ for $n\geq 1$. We identify $x^n$ in cosimplicial degree zero with $x^n \otimes 1 \in U(\symcofree(X)) \otimes  \Lambda(\Sigma^{-1}X)$ and $\sum_{i=1}^{n} i \cdot x^{n-i} \otimes x^i$ with $x^{n-1}\otimes \sigma^{-1}x \in U(\symcofree(X)) \otimes  \Lambda(\Sigma^{-1}X).$

Now assume that $X$ is concentrated in odd degree and that the characteristic of $k$ is not $2$, so that $S^c(X) = \Lambda_k(x)$.
Hence a typical element in $\coTHH^a(S^c(X))$ is of the form 
$y=x^{t_0} \otimes x^{t_1} \otimes \dotsb \otimes x^{t_a}$ with $t_i \in \lbrace 0,1 \rbrace$. Now $y\in N^a(\coTHH^{\bullet}(S^c(X)))$ if and only if $t_1=\dotsb=t_a = 1$, and all differentials in $N^*(\coTHH^{\bullet}(S^c(X)))$ are trivial. Identifying $1 \otimes x^{\otimes a}$ with $1 \otimes (\sigma^{-1} x)^a \in U(\symcofree(X)) \otimes U(\symcofree(\Sigma^{-1}X))$ and $x \otimes x^{\otimes a}$ with $x \otimes (\sigma^{-1}x)^{\otimes a} \in U(\symcofree(X)) \otimes U(\symcofree(\Sigma^{-1}X))$ yields the result.
\end{proof}

The general case follows from the interplay of products and the cofree functor and the following result.

\begin{lemma}[Cf. {\cite[p. 719]{michaelis}}]\label{lem:infDimCoalg} 
Let $k$ be a field and let $X$ be a graded $k$-module. Then
\[\symcofree(X) \cong \colim_{V\subset X \text{\,fin\,dim}}  \symcofree(V),\]
where the canonical projection $\pi_X \colon \symcofree(X)\cong\colim_{V\subset X \text{fin\,dim}}  S^c(V) \ra X$ is given by the colimit of the maps  $\symcofree(V)\xto{\pi_V} V\to X$.  
\end{lemma}

\begin{proof}[Proof of Theorem~\ref{thrm:HKRcofreedg}]
Recall from Example \ref{ex:ch} that $\coTHH(C)$ of a counital coalgebra in $\dgm^{\geq 0}$ can be computed as
\[\coTHH_{\otimes}(C) = \Tot_{\dgm} (N^* \coTHH^{\bullet}_{\otimes}(C)).\]

Assume first that $X$ has trivial differential. If $X$ is finite dimensional,
\[\coTHH^{\bullet}_{\otimes}(S^c(X)) \cong \coTHH^{\bullet}_{\otimes}(S^c(X_1)) \otimes ... \otimes \coTHH^{\bullet}_{\otimes}(S^c(X_n))\]  
for one dimensional $k$-vector spaces $X_i$. Hence Lemma \ref{lem:HKROneGenerator} and the dual Eilenberg--Zilber map yield the desired quasi-isomorphism. Lemma \ref{lem:infDimCoalg} proves Theorem \ref{thrm:HKRcofreedg} for infinite dimensional cochain complexes $X$ with trivial differential.

If $X$ is any nonnegatively graded cochain complex,  applying the result for graded $k$-modules which we just proved 
shows that there is a morphism of double complexes
\[N^*(\coTHH^{\bullet}_{\otimes}(S^c(X))) \ra \Omega^{\symcofree(X)\vert k}\]
which is a quasi-isomorphism on each row, that is,  if we fix the degree induced by the grading on $X$. Hence this induces a quasi-isomorphism on total complexes.
\end{proof}

\begin{rem} \label{rem:KaehlerFS}
Farinati and Solotar \cite[Section 3]{farinati-solotar-ext}  define a symmetric $C$-bicomodule $\Omega^1_C$ and a coderivation $d \colon \Omega^1_C \ra C$  for any (ungraded) coaugmented coalgebra $C$ such that $(\Omega^1_C, d)$ satisfies the following universal property: Every coderivation $f\colon M \ra C$ from a symmetric $C$-bicomodule $M$ to $C$ factors as 
$$f= d \circ \tilde f$$
for a unique $C$-bicomodule morphism $\tilde f$. Farinati and Solotar also give a construction of $(\Omega^1_C, d)$ dual to the construction of the module of K\"ahler differentials associated to a commutative algebra.
To compare this with our definition of $\Omega^{\symcofree(X)\vert k}$, note that for $X$ concentrated in degree zero
$$\Omega^1_{\symcofree(X)} \cong \symcofree(X) \otimes X$$
as a $\symcofree(X)$-bicomodule. Since $\symcofree(X)$ is cofree, coderivations into $\symcofree(X)$ correspond to maps into $X$, and the coderivation $d \colon \symcofree(X) \otimes X \ra \symcofree(X)$  is induced by the map $\epsilon \otimes \id \colon \symcofree(X) \otimes X \ra X$.
If the characteristic of $k$ is different from $2$, this yields that $\Omega^{\symcofree(X) \vert k}$ coincides with the exterior coalgebra $\Lambda_{\symcofree(X)}(\Omega^1_{\symcofree(X)})$ on $\Omega^1_{\symcofree(X)}$ defined in \cite[Section 6]{farinati-solotar-ext}.
\end{rem}

\begin{rem}
Analogous to Proposition \cite[5.4.6]{loday}, the proof of Theorem \ref{thrm:HKRcofreedg} shows that the Theorem actually holds for any differential graded cocommutative couaugmented coalgebra $C$ whenever the underlying graded cocommutative coaugmented coalgebra is cofree.
\end{rem}

\begin{rem}
A similar result holds for unbounded chain complexes if we define $\Tot_{\dgm}$ and hence $\coTHH$ as a direct sum instead of a product in each degree. For nonnegatively graded cochain complexes both definitions of $\Tot_{\dgm}$ agree. However, if we use the homotopically correct definition of $\coTHH$ for unbounded chain complexes via the product total complex, Lemma \ref{lem:HKROneGenerator} doesn't hold for vector spaces $X$ that are concentrated in degree $-1$.
\end{rem}

\section{A coB\"okstedt Spectral Sequence}\label{sect:thespectralsequence}

While in algebraic cases we are able to understand certain good examples of   $\coTHH$ by an analysis of the definition, in topological examples we require further tools.  In this section, we construct a spectral sequence, which we call the {\em coB\"okstedt spectral sequence.}  {Recall that for a field $k$ and a ring spectrum $R$, the skeletal filtration on the simplicial spectrum $\THH(R)_{\bullet}$ yields a spectral sequence 
\[
E^2_{*,*} = \HH_*(H_*(R;k)) \Rightarrow H_*(\THH(R);k),
\] called the B\"okstedt spectral sequence \cite{Bo2}. Analogously, for a coalgebra spectrum $C$, we consider the Bousfield--Kan spectral sequence arising from the cosimplicial spectrum $\coTHH^{\bullet}(C)$.  We show in Theorem \ref{sscoalgebra} that this is a spectral sequence of coalgebras.   Since our spectral sequence is an instance of the Bousfield--Kan spectral sequence, its convergence is not immediate, but in the case where $C$ is a suspension spectrum $\Sigma_{+}^{\infty} X$ with $X$ a simply connected space, {we provide conditions  in Corollary~\ref{cor.conv.2} below under which this spectral sequence converges to }$H_{*}(\coTHH(\Sigma_{+}^{\infty} X);k).$

Let $(\Spec, \wedge , \mathbb{S})$ denote a symmetric monoidal category of spectra, such as those given by \cite{ekmm}, \cite{hss}, or \cite{mmss}. The notation $\sma$ means ``smash product over $\mathbb{S}$.'' Let $C$ be a coalgebra in this category with comultiplication $\comult\colon C \to C\wedge C$ and counit $\varepsilon\colon C \to \mathbb{S}$. Assume that $C$ is cofibrant as a spectrum so that $\coTHH(C)$ has the correct homotopy type, as in Remark \ref{rem:fibrancyassumptionsforholim}. Note that the spectral homology of $C$ with coefficients in a field $k$ is a graded $k$-coalgebra with structure maps 
\begin{align*}
\comult\colon  H_*(C; k) &\xrightarrow{H_*(\comult)} H_*(C \wedge C; k)\cong  H_*(C; k)\otimes_{k} H_*(C; k), \\
\varepsilon\colon H_*(C; k) &\xrightarrow{H_*(\varepsilon)} k .
\end{align*}

\begin{thm}\label{thm:SpecSeq} Let $k$ be a field. Let $C$ be a coalgebra spectrum that is cofibrant as a spectrum. The Bousfield--Kan spectral sequence for the cosimplicial spectrum $\coTHH^\bullet(C)$ gives a \emph{coB\"okstedt spectral sequence}  for calculating $H_{t-s}(\coTHH(C);k)$ with  $E_2$-page 
\[ E_2^{s, t}=\coHH^{k}_{s,t}(H_*(C;k)).\]
given by the classical coHochschild homology of $H_*(C;k)$ as a graded $k$-module.
\end{thm}

\begin{proof}
The spectral sequence arises as the Bousfield--Kan spectral sequence of a cosimplicial spectrum. We briefly recall the general construction. Let $X^\bullet$ be a Reedy fibrant cosimplicial spectrum and recall that 
\[\Tot (X^{\bullet}) = \text{eq}\Big(
\prod_{n \geq 0} (X^n)^{\Delta^n} \rightrightarrows \prod_{\alpha \in \Delta([a],[b])} (X^b)^{\Delta^a}
\Big),\]
where $\Delta$ is the cosimplicial space whose $m$-th level is the standard $m$-simplex $\Delta^m$. Let $\sk_n \Delta \subset \Delta$ denote the cosimplicial subspace whose $m$-th level is $ sk_n \Delta^m$, the $n$-skeleton of the $m$-simplex and set
\[\Tot_n (X^{\bullet}) := \text{eq}\Big(
\prod_{m \geq 0} (X^m)^{\sk_n \Delta^m} \rightrightarrows \prod_{\alpha \in \Delta([a],[b])} (X^b)^{\sk_n \Delta^a}
\Big).\]

 The inclusions $sk_n \Delta \hookrightarrow sk_{n+1} \Delta$ induce the maps in a tower of fibrations
\begin{align*}
\dotsb \to \Tot_n(X^\bullet) \xrightarrow{p_n} \Tot_{n-1}(X^\bullet)  \xrightarrow{} \dotsb \to \text{Tot}_0(X^\bullet) \cong X^0.
\end{align*}
Let $F_n \xrightarrow{i_n} \text{Tot}_n(X^\bullet) \xrightarrow{p_n} \text{Tot}_{n-1}(X^\bullet)$ denote the inclusion of the fiber and consider the associated exact couple: 
\[\begin{tikzpicture}
\node(a){$\pi_*(\Tot_*(X^\bullet) ) $};
\node(b)[right of=a, node distance = 3cm]{$\pi_*(\Tot_*(X^\bullet))$};
\node(c)[below of= a,node distance = 1.5cm, right of=a, node distance = 1.5cm]{$\pi_*(F_*).$};

\draw[->](a) to node [above]{$p_*$} (b);
\draw[->](b) to node [below right]{$\partial$} (c);
\draw[->](c) to node [below left]{$i_*$} (a);
\end{tikzpicture}\]
This exact couple gives rise to a cohomologically graded spectral sequence $\{E_r, d_r\}$ with $E_1^{s,t}=\pi_{t-s}(F_{s})$ and differentials $d_r\colon E_r^{s,t} \to E_r^{s+r,t+r-1}$. It is a half plane spectral sequence with entering differentials.

The fiber $F_{s}$ can be identified with $\Omega^s (N^s X^\bullet)$, where
\[N^sX^\bullet=\eq\Big( X^s\xrightrightarrows[*]{(\sigma^0,\dots, \sigma^{s-1})}\prod_{i=0}^{s-1} X^{s-1}\Big)\]
We have isomorphisms 
\[E_1^{s,t} =\pi_{t-s}(\Omega^s(N^s X^\bullet)) \cong \pi_{t}(N^s X^\bullet) \cong N^s \pi_t(X^\bullet)\]
 and under these isomorphisms the differential $d_1\colon N^s \pi_t(X^\bullet) \to N^{s+1} \pi_t(X^\bullet)$ is identified with $\sum (-1)^i \pi_t(\delta^i)$. Since the cohomology of the normalized complex agrees with the cohomology of the cosimplicial object, we conclude that the $E_2$ term is
\[E_2^{s,t}\cong H^s(\pi_t(X^\bullet), \textstyle\sum (-1)^i \pi_t(\delta^i)).\]

Let $C$ be a coalgebra spectrum and consider the spectral sequence arising from the Reedy fibrant replacement $R(\coTHH^{\bullet}(C)\wedge H k)$. We have isomorphisms 
\[
\pi_* R(\coTHH^{n}(C)\wedge H k) \cong \pi_*(\coTHH^{n}(C) \wedge H k) \cong H_*(C;k)^{\otimes_{k}(n+1)}
 \]
The map $\pi_*(\delta^i)$ corresponds to the $i$'th coHochschild differential under this identification, therefore
\[E_2^{s,t}= \mathrm{coHH}^{k}_{s,t}(H_*(C;k)). \qedhere\]
\end{proof}

From~\cite[IX 5.7]{BK}, we have the following statement about the convergence of the coB\"okstedt spectral sequence. {See also~\cite[VI.2]{goerss-jardine} for a discussion of complete convergence.}  In~\cite[IX 5]{BK}, the authors restrict to $i \geq 1$ so that they are working with groups.  Since we have abelian groups for all $i$, that restriction is not necessary here.  

\begin{prop}\label{prop:bkconvergence}
If for each $s$ there is an $r$ such that $E_r^{s, s+i} = E_{\infty}^{s, s+i}$, then the coB\"okstedt spectral sequence for $\coTHH(C)$ converges completely to \[\pi_* \Tot R(\coTHH^\bullet(C) \wedge Hk).  \]
\end{prop}

There is a natural map
\[P\colon \Tot(R\coTHH^\bullet(C)) \wedge Hk \to \Tot R(\coTHH^\bullet(C) \wedge Hk)\]
 arising from the natural construction of a map of the form $\Hom(X,Y)\sma Z \to \Hom(X,Y\sma Z)$. 
  Since $\pi_*(\Tot(R\coTHH^\bullet(C)) \wedge Hk) \cong H_{*}(\coTHH(C);k),$ we have the following.

\begin{cor}\label{cor.conv.1}
If the conditions on $E_r^{s, s+i}$ in Proposition \ref{prop:bkconvergence} hold and the map $P$ described above induces an isomorphism in homotopy, then the coB\"okstedt spectral sequence for $\coTHH(C)$ converges completely to $H_{*}(\coTHH(C);k)$.  
\end{cor}

\begin{ex}\label{ex:freeloopspace}
Let $X$ be a simply connected space, let $d\colon X \to X \times X$ denote the diagonal map and let $c\colon X \to *$ be the map collapsing $X$ to a point. We form a cosimplicial space $X^{\bullet}$ with
$X^n = X^{\times (n+1)}$ and with cofaces
\[
\delta_i \colon X^{\times (n+1)} \ra X^{\times (n+2)}, \quad \delta_i=\begin{cases}
X^{\times i} \times d \times X^{\times (n-i)}, & 0\leq i \leq n,\\
{\tau_{X, X^{\times (n+1)}}}(d \times  X^{\times n}), & i=n+1,
\end{cases}\]
and codegeneracies
\[
\sigma_i \colon X^{\times (n+2)} \ra X^{\times (n+1)}, \quad
\sigma_i = X^{\times (i+1)} \times c \times X^{\times (n-i)} \quad \text{for } 0 \leq i \leq n.
\]
The  cosimplicial space $X^\bullet$ is the cosimplicial space $\coTHH^\bullet(X)$  of Definition \ref{defn:coTHH} for the coalgebra $(X,d,c)$ in the category of spaces; see Example \ref{ex:simpset}. As  noted earlier in Proposition \ref{prop:coTHHcotensorS1simpmodelcat}, the cosimplicial space $X^\bullet$ totalizes to the free loop space $\mathcal{L}X$; see Example 4.2 of \cite{bousfield}. 

We can consider the spectrum $\Sigma_+^{\infty} X$ as a coalgebra with comultiplication arising from the diagonal map $d\colon X \to X \times X$ and counit arising from $c\colon X \to *$. We have an identification of cosimplical spectra
\[
\coTHH^{\bullet}(\Sigma_+^{\infty} X)=\Sigma^{\infty}_+( X^{\bullet}).
\]
The topological coHochschild homology is the totalization of a Reedy fibrant replacement of the above cosimplicial spectrum. We have a natural map
\[
\Sigma^{\infty}_+ \Tot(X^{\bullet}) \to \Tot(\Sigma^{\infty}_+ X^{\bullet}),
\]
as Malkiewich describes in {Section 2 of \cite{malkiewich}}, which after composing with the map to the totalization of a Reedy fibrant replacement becomes a stable equivalence {for simply connected $X$; see Proposition 2.22} of \cite{malkiewich}. Hence we have a stable equivalence
\[
\Sigma^\infty_+ \mathcal{L}X \xrightarrow{\cong}  \coTHH(\Sigma_+^{\infty} X).
\]
\end{ex}

\begin{cor}\label{cor.conv.2}
{Let $X$ be a simply connected space.  If for each $s$ there is an $r$ such that $E_r^{s, s+i} = E_{\infty}^{s, s+i}$, the coB\"okstedt spectral sequence arising from the coalgebra $\Sigma_+^{\infty} X$ converges completely to 
\[H_{*}(\coTHH(\Sigma_+^{\infty} X);k) \cong H_*(\mathcal{L}X; k).\]
}
\end{cor}

\begin{proof}
Since $X$ is simply connected, by \cite[2.22]{malkiewich}
$H_{*}(\coTHH(\Sigma_+^{\infty} X);k) \cong H_*(\mathcal{L}X; k).$   As in Corollary~\ref{cor.conv.1}, we need to show that there is a weak equivalence
\begin{equation}\label{TotCompare} \Tot (R\coTHH^\bullet(\Sigma_+^{\infty} X)) \wedge Hk \simeq \text{Tot}R(\coTHH^\bullet(\Sigma_+^{\infty} X) \wedge Hk).\end{equation}
First we consider the $n$th stage approximations:
    \[\Tot_n (R\coTHH^\bullet(\Sigma_+^{\infty} X)) \wedge Hk \simeq \text{Tot}_n R(\coTHH^\bullet(\Sigma_+^{\infty} X) \wedge Hk).\]   
    The $n$th stages are weakly equivalent because the derived $\Tot_n$ functor is a finite homotopy limit functor.  In spectra, this agrees with a finite homotopy colimit functor and hence the derived $\Tot_n$ commutes with smashing with $Hk$. Since these weak equivalences are compatible for all $n$, it follows that the inverse limits over $n$ of both towers are weakly equivalent.  On the right hand side, this inverse limit agrees with the right hand side of Equation \ref{TotCompare}. For the left hand sides to agree, we need to commute the inverse limit with $- \wedge Hk$.  This is possible here because of the properties of this specific tower.
    
    In~\cite[2.22]{malkiewich}, and in more detail in an {earlier version~\cite[4.6]{malkiewichv1}}, Malkiewich shows that the connectivity of the fibers of the tower $ \{ \Tot_n R(\coTHH^\bullet(\Sigma_+^{\infty} X)) \}$ tend to infinity, and hence there is a pro-weak equivalence between this tower and the constant tower given by $\{\Tot R(\coTHH^\bullet(\Sigma_+^{\infty} X))\}$.  That is, applying homotopy, $\pi_k$, produces a pro-isomorphism of groups for each $k$.  By~\cite[8.5]{bousfield}, it follows that there is a pro-isomorphism for each $l$ between the constant tower   $\{H_l\Tot(R\coTHH^\bullet(\Sigma_+^{\infty} X))\}$ and the tower $ \{ H_l\Tot_n R(\coTHH^\bullet(\Sigma_+^{\infty} X)) \}$.  It follows that there is a pro-weak equivalence 
    \[ \{\Tot(R\coTHH^\bullet(\Sigma_+^{\infty} X)) \wedge Hk \}  \xrightarrow{\simeq}  \{\Tot_n(R\coTHH^\bullet(\Sigma_+^{\infty} X)) \wedge Hk  \}.\]  
    By~\cite[III.3.1]{BK}, it follows that the inverse limits are weakly equivalent.  That is, there is a weak equivalence
     \[\Tot R(\coTHH^\bullet(\Sigma_+^{\infty} X)) \wedge Hk \xrightarrow{\simeq}  \lim_n(\Tot_n R(\coTHH^\bullet(\Sigma_+^{\infty} X)) \wedge Hk).\]
     This shows that $- \wedge Hk$ commutes with inverse limits here as desired.
\end{proof}

Our next result is to prove that the coB\"okstedt spectral sequence of Theorem \ref{thm:SpecSeq} is a spectral sequence of coalgebras.  We exploit this additional structure in Section \ref{sect:computations} to make calculations of topological coHochschild homology in nice cases.

 We first describe a functor $\Hom(-,C)\colon \Set^\op\to \comcoalg$, where $C$ is a cocommutative coalgebra. In particular this will give us a cosimplicial spectrum $\textup{Hom}(S^1_{\bullet},C)$ that agrees with the cosimplicial spectrum $\coTHH^{\bullet}(C)$ defined in Definition \ref{defn:coTHH}.

Let $X$ be a set and $C$ be a cocommutative coalgebra spectrum.  Let $\Hom(X,C)$ be the indexed smash product
\[ \bigwedge_{x\in X} C.\]
If $f\colon X\to Y$ is a map of sets, we define a map of spectra $\Hom(Y,C)\to \Hom(X,C)$ as the product over $Y$ of the component maps
\[ C\to \bigwedge_{x\in f^\inv(y)} C\]
given by iterated comultiplication over $f^\inv(y)$. Here by convention the smash product indexed on the empty set is $\mathbb{S}$, and the map to $C\to \mathbb{S}$ is the counit map. Since $C$ is cocommutative, this map doesn't depend on a choosen ordering for applying the comultiplications and for  each $X$, the comultiplication on $C$ extends to define a cocommutative comultiplication on $\Hom(X,C)$.  Hence $\Hom(-,C)$ is a functor $\Set^\op\to \comcoalg$.  If $X_*$ is a simplicial set, the composite 
\[ (\Delta^\op)^\op \xto{X_*^\op}\Set^\op \xto{\Hom(-,C)} \comcoalg\]
thus defines a cosimplicial cocommutative coalgebra spectrum.  

To obtain the cosimplicial coalgebra yielding $\coTHH(C)$, we take $X_*$ to be the simplicial circle $\Delta^1_{\bullet}/\partial\Delta^1_{\bullet}$.  Here $\Delta^1_{\bullet}$ is the simplicial set with $\Delta^1_n=\{x_0,\dotsc,x_{n+1}\}$ where $x_t$ is the function $[n]\to [1]$ so that the preimage of $1$ has order $t$.  The face and degeneracy  maps are given by:
\[ \delta_i(x_t)= \begin{cases} x_t & t\leq i\\ x_{t-1}& t>i \end{cases}\]
\[ \sigma_i(x_t)=\begin{cases}x_t & t\leq i\\ x_{t+1} & t>i\end{cases}\]
The $n$-simplices of $S^1_*$ are obtained by identifying $x_0$ and $x_{n+1}$ so that $S^1_n=\{x_0,\dotsc, x_n\}$.  Comparing the coface and codegeneracy maps in $\Hom(S^1_{\bullet}, C)$ and $\coTHH^\bullet(C)$ shows that the two cosimplicial spectra are the same.

\begin{thm}\label{sscoalgebra}
Let $C$ be a connected cocommutative coalgebra that is cofibrant as a spectrum. Then the Bousfield--Kan spectral sequence described in Theorem \ref{thm:SpecSeq} is a spectral sequence of $k$-coalgebras. In particular, for each $r>1$ there is a coproduct
\[
\psi\colon E_r^{**} \rightarrow E_r^{**} \otimes_{k} E_r^{**},
\]
and the differentials $d_r$ respect the coproduct. 
\begin{proof}

Below we will construct the coproduct using natural maps of spectral sequences
\[
E_r^{**} \xto{\ \nabla\ }  D_r^{**}  \xleftarrow{\ AW\ }   E_r^{**} \otimes_{k} E_r^{**}
\]
where the map $AW$ is an isomorphism for $r\geq 2$. 

Recall that the cosimplicial spectrum $\coTHH^\bullet(C)$ can be identified with the cosimplicial spectrum Hom($S^1_{\bullet}, C)$. Let $D_r^{**}$ denote the Bousfield--Kan spectral sequence for the cosimplicial spectrum $\Hom((S^1\amalg S^1)_{\bullet}, C)\sma Hk$. The codiagonal map $\nabla\colon S^1\amalg S^1\to S^1$ is a simplicial map and gives a map $\Hom(S^1_{\bullet},C)\to \Hom((S^1\amalg S^1)_{\bullet},C).$ The map $\nabla\colon  E_r^{**} \rightarrow D_r^{**}$ is induced by this codiagonal map. Let $X^{\bullet}$ denote the cosimplicial $Hk$-coalgebra spectrum $\coTHH(C)^{\bullet} \sma Hk$. Note that the cosimplicial spectrum Hom($(S^1\amalg S^1)_{\bullet}, C) \sma Hk$ is the cosimplicial spectrum $[n] \to X^n \sma_{Hk} X^n$. This is the diagonal cosimplicial spectrum associated to a bicosimplicial spectrum, $(p, q) \to X^p \sma_{Hk} X^q$. We will denote this diagonal cosimplicial spectrum by $\diag(X^{\bullet} \sma_{Hk} X^{\bullet}$). The spectral sequence $D_r^{* *}$ is given by the Tot tower for the Reedy fibrant replacement $R (\diag(X^{\bullet} \sma_{Hk} X^{\bullet}$)). 

Note that our cosimplicial spectra are in fact cosimplicial $Hk$-modules. We can apply the standard equivalence between $Hk$-modules and chain complexes of $k$-modules \cite{shipley-hz}. By hypothesis, our $Hk$-modules are connective and therefore can be replaced by non-negatively graded chain complexes of $k$-modules. By the Dold--Kan correspondence, non-negatively graded chain complexes of $k$-modules are equivalent to simplicial $k$-modules. This also holds on the level of diagram categories and therefore our cosimplicial $Hk$-module spectra can be identified as cosimplicial simplicial $k$-modules. This string of equivalences is weakly monoidal in the sense of \cite{schwede-shipley-equiv}.

The Bousfield--Kan results in the simplicial setting \cite{ BKpairings, BKquadrant} thus apply, giving us a map of spectral sequences
\[
E_r^{**} \otimes_k E_r^{**} \to D_r^{**}.
\] 
By Bousfield and Kan, on the $r=1$ page, this map is given by the Alexander--Whitney map:
\[
 E_1^{**} \otimes_{k} E_1^{**} = N^*( \pi_*X^{\bullet}) \otimes_{k} N^*(\pi_*X^{\bullet}) \rightarrow N^*( \pi_*X^{\bullet} \otimes_{k} \pi_*X^{\bullet}) =  D_1^{**}.
\] 

By the cosimplicial analog of the Eilenberg--Zilber theorem (see, for example, \cite{fresse}) this induces an isomorphism on cohomology 
\[
 \xymatrix{
H^*(E_1^{*,*}) \otimes_{k} H^*(E_1^{*,*}) \cong H^*(E_1^{*,*}\otimes_{k} E_1^{*,*}) \ar[r]^-{\cong}&  H^*(D_1^{*,*}) \cong D_{2}^{**} ,
 }
 \]
 where the first isomorphism is given by the K\"unneth theorem. 
 This composite is the map $
 E_{2}^{*,*} \otimes_{k} E_{2}^{*,*}\xto{\cong}  D_{2}^{**} .
$
 By induction and repeated application of the K\"unneth theorem, we get equivalences  
 $ E_{r}^{*,*} \otimes_{k} E_{r}^{*,*} \xto{\cong}  D_{r}^{**}$
for all $r>1$. The composite of $\nabla$ with the inverse of this isomorphism gives us the desired coproduct.
\end{proof}
\end{thm}

\section{Computational results}\label{sect:computations}

Now that we have developed the coB\"okstedt spectral sequence and its coalgebra structure, we use these structures to make computations of the homology of $\coTHH(C)$ for certain coalgebra spectra $C$. We also prove that in certain cases this spectral sequence collapses.

We first make  some  elementary computations of coHochschild homology, which we will later use as input to the coHH to coTHH spectral sequence described in Theorem \ref{thm:SpecSeq}. The coalgebras we consider are the underlying coalgebras of Hopf algebras, specifically, of exterior Hopf algebras, polynomial Hopf algebras and divided power Hopf algebras. As mentioned in Example \ref{ex:permutfriendlythings}, the exterior Hopf algebra and the divided power Hopf algebras give the free cocommutative coalgebras on graded $k$-modules concentrated in odd degree or even degree, respectively. Recall from Example \ref{ex:permutfriendlythings} that $\Lambda_{k}(y_1, y_2, \dotsc)$ denotes the exterior Hopf algebra on generators $y_i$ in odd degrees and that $\Gamma_{k}[x_1, x_2, \dotsc]$ denotes the divided power Hopf algebra on generators $x_i$ in even degrees. Let $k[w_1,w_2,\dotsc]$ denote the polynomial Hopf algebra on generators $w_i$ in even degree.  The coproduct is given by $\comult(w_i^j)=\sum_{k} \binom{j}{k}\,w_i^k\otimes w_i^{j-k}.$

In the following computations, we assume that all the Hopf algebras in question  only have finitely many generators in any given degree. We work over a field $k$. 
The following computation of coHochschild homology is the main result we use as input for our spectral sequence. 

\begin{prop}\label{cohhcomp}
Let the cogenerators $x_i$ be of even nonnegative 
 degree and let the cogenerators $y_i$ be of odd nonnegative degree.  We evaluate the coHochschild homology as a coalgebra for the following cases of free cocommutative coalgebras: 
\[
\coHH_*(\Gamma_{k}[x_1, x_2, \dotsc]) = \Gamma_{k}[x_1, x_2, \dotsc] \otimes \Lambda_{k}(z_1, z_2, \dotsc),
\]
where $\deg(z_i) = (1, \deg(x_i))$. 
\[
\coHH_*(\Lambda_{k}(y_1, y_2, \dotsc)) = \Lambda_{k}(y_1, y_2, \dotsc) \otimes k[w_1, w_2, \dotsc] 
\]
where $\deg(w_i) = (1, \deg(y_i))$.

\end{prop}

Before proving the proposition, we recall some basic homological coalgebra. Consider any  coalgebra $C$  over a field $k$. Let $M$ be a right $C$-comodule and $N$ be a left $C$-comodule with corresponding maps $\rho_M\colon M \rightarrow M \otimes C$ and $\rho_N\colon N \rightarrow C \otimes N$. Recall that the cotensor product of $M$ and $N$ over $C$, $M \Box_C N$, is defined as the kernel of the map
\[
\rho_M \otimes 1 - 1 \otimes \rho_N\colon M \otimes N \to M \otimes C \otimes N. \]
The Cotor functors are the derived functors of the cotensor product. In other words, $\cotor_C^n(M, N) = H^n(M \Box_C \bf{X})$ where $\bf{X}$ is an injective resolution of $N$ as a left $C$-comodule.

  Let $C^e = C \otimes C^{op}$. Let $N$ be a $(C, C)$-bicomodule. Then $N$ can be regarded as a right $C^e$-comodule. As shown in Doi \cite{doi}, $\coHH_n(N, C) = \cotor^n_{C^e}(N, C)$. 

\begin{lemma}

For $C = \Gamma_{k}[x_1, x_2, \ldots]$ or $C= \Lambda_{k}(y_1, y_2, \ldots)$,
\[\coHH(C) = C \otimes \cotor_C(k,k).\]
\end{lemma}

\begin{proof}
As above, $\coHH_n(C) = \cotor^n_{C^e}(C, C)$. Observe that $C$ is a Hopf algebra. Let $S$ denote the antipode in $C$. As in Doi \cite[Section 3.3]{doi}, we consider the coalgebra map
\[
\nabla\colon C^e \rightarrow C
\]
given by $\nabla(c \otimes d^{op}) = cS(d)$. Given a $(C, C)$-bicomodule $M$, let $M_{\nabla}$ denote $M$ viewed as a right $C$-comodule via the map $\nabla$. 

Let $k \rightarrow Y^0 \rightarrow Y^1 \rightarrow \cdots$ be an injective resolution of $k$ as a left $C$-comodule. When $C$ is a Hopf algebra, $(C^e)_{\nabla}$ is free as a right $C$-comodule, and  $(C^e)_{\nabla}\Box_C Y^n$ is injective as a left $C^e$-comodule \cite{doi}. The sequence  
\[
(C^e)_{\nabla}\Box_Ck \rightarrow (C^e)_{\nabla}\Box_C Y^0 \rightarrow(C^e)_{\nabla}\Box_C Y^1 \rightarrow \cdots
\]
is therefore an injective resolution of $(C^e)_{\nabla}\Box_Ck \cong C$ as a left $C^e$-comodule. Cotensoring over $C^e$ with $C$ we have:
\[
C \Box_{C^e}((C^e)_{\nabla}\Box_C Y^0) \rightarrow C \Box_{C^e}((C^e)_{\nabla}\Box_C Y^1) \rightarrow \cdots
\]
which gives:
\[
C_{\nabla}\Box_C Y^0 \rightarrow C_{\nabla}\Box_C Y^1 \rightarrow \cdots
\]
Therefore $\coHH_n(C) \cong \cotor^n_{C^e}(C, C) \cong \cotor^n_C(C_{\nabla}, k).$ In our cases $C$ is cocommutative, so $C_{\nabla}$ has the trivial $C$-comodule structure. Therefore $\cotor^n_C(C_{\nabla},k) \cong C \otimes \cotor^n_C(k, k).$ 
\end{proof}

We now want to compute $\cotor_C(k, k)$ for the coalgebras $C$ that we are interested in. 
\begin{prop}
For $C = \Gamma_{k}[x_1, x_2, \ldots]$, $\cotor_C(k,k) =\Lambda_{k}(z_1, z_2, \ldots)$ where $\deg(z_i) = (1, \deg(x_i))$. For $C = \Lambda_{k}(y_1, y_2, \ldots)$, $\cotor_C(k,k) = k[w_1, w_2, \ldots] $ where $\deg(w_i) = (1, \deg(y_i))$.
\end{prop}

\begin{proof}
Note that in both cases $C$ is a cocommutative coalgebra that in each fixed degree is a free and finitely generated  $k$-module. Let $A$ denote the Hopf algebra dual of $C$. As in \cite{neisendorfer} we conclude that $\cotor_C(k, k)$ is a Hopf algebra, which is dual to the Hopf algebra $\tor^A(k,k)$. 

Recall that the Hopf algebra dual of $\Gamma_{k}[x_1, x_2, \dotsc]$ is $k[x_1, x_2, \dotsc]$ and the Hopf algebra $\Lambda_{k}(y_1, y_2, \dotsc)$ is self dual. We recall the classical results:
\begin{align*}
\tor^{k[x_1, x_2, \ldots]}(k,k) &\cong \Lambda_{k}(z_1, z_2, \ldots),\\
\tor^{\Lambda_{k}(y_1, y_2, \ldots)}(k,k) &\cong \Gamma_{k}[w_1, w_2, \ldots]
\end{align*}
where the degree of $z_i$ is $(1, \deg(x_i))$ and the degree of $w_i$ is  $(1, \deg(y_i))$. The result follows. 
\end{proof}

We have now proven Proposition \ref{cohhcomp}.  These results provide the input for the coB\"okstedt spectral sequence.  They are the starting point for the following theorem, which says that this spectral sequence collapses at $E_2$ in the case where the homology of the input is cofree on cogenerators of even nonnegative degree.

\begin{thrm}\label{comp1} Let $C$ be a cocommutative coassociative coalgebra spectrum that is cofibrant as a spectrum, and whose homology coalgebra is  
$$H_*(C;k) = \Gamma_{k}[x_1,x_2,\dots],$$ 
where the $x_i$ are cogenerators in nonnegative even degrees and there are only finitely many cogenerators in each degree. Then the coB\"okstedt spectral sequence for $C$ collapses at $E_2$, and 

\[ E_2\cong E_{\infty} \cong \Gamma_{k}[x_1, x_2, \dots]   \otimes \Lambda_{k}(z_1, z_2, \dots) ,\]
with $x_i$ in degree $(0, \deg(x_i))$ and $z_i$ in degree $(1, \deg(x_i))$.

\end{thrm}

\begin{proof}
By Proposition \ref{cohhcomp}, 
\[\coHH_*(\Gamma_{k}[x_1, x_2, \dots] ) = \Gamma_{k}[x_1, x_2, \dots]  \otimes \Lambda_{k}(z_1, z_2, \dots).\]

A conilpotent coalgebra $C$ in graded $k$-modules is called \emph{cogenerated} by $Y$ if there is a surjection $\pi \colon C \ra Y$ such that
\[ C \xto{(\comult^{n})_{n}}  \cofree(C) \xto{\cofree(\pi)}  \cofree(Y)\]
is injective. Every cofree cocommutative coalgebra $\symcofree(X)$ is cogenerated by $X$ under the canonical projection $\pi\colon \symcofree(X) \to X$, since the composite
\[\symcofree(X) \xto{(\Delta^{n})_{n}}  \cofree(\symcofree(X)) \xto{\cofree(\pi)}  \cofree(X)\]
is the inclusion of $\symcofree(X)$ into $\cofree(X)$.

Let $p$ be the characteristic of $k$. If $p\neq 2$, $\coHH_*(\Gamma_{k}[x_1, x_2, \dots] )$ is the cofree cocommutative conilpotent coalgebra cogenerated by $\lbrace x_1, x_2, \dots, z_1,z_2, \dots \rbrace$ with respect to the total grading. If $p=2$, it is easy to verify that $\Gamma_{k}[x_1, x_2, \dots]  \otimes \Lambda_{k}(z_1, z_2, \dots)$ is cogenerated by $\lbrace x_1, x_2, \dots, z_1,z_2, \dots \rbrace$ as well.

If $C$ is cogenerated by $\pi\colon C \ra Y$, any coderivation $d \colon C \ra C$ is completely determined by $\pi d$: Since 
\[\comult^{n} d  = \sum_{i=1}^{n+1}C^{\otimes i-1} \otimes  d \otimes C^{\otimes n+1-i}\circ\comult^{n}\]
and hence 
\[ \cofree(\pi) (\comult^{n})_n d = (\sum_{i=1}^{n+1}\pi^{\otimes i-1} \otimes  \pi d \otimes \pi^{\otimes n+1-i}\comult^{n})_n,\]
the injectivity of $\cofree(\pi) (\comult^{n})_n$ yields that $\pi d$ determines $d$.
 In particular, if $\pi d =0$, this implies $d=0$.

We now return to the coB\"okstedt spectral sequence. Assume that for some $r \geq 2$ we already know that $d^2,\dotsc, d^{r-1}$ vanish and thus the pages $E_2 \cong \dotsb \cong E_{r-1}$ are as claimed. The differential $d^r$ has bidegree $(r, r-1)$. Since the cogenerators are of bidegrees $(0,t)$ and $(1,t)$ and $E^{s,t}_{r-1} =0$ if $s<0$, the cogenerators cannot be in the image of $d^r$ and we find that $\pi d^r =0$  if $ r\geq 2$. 
\end{proof}

Recall that for a space $X$ the suspension spectrum $\Sigma^{\infty} X$ is an $\mathbb{S}$-coalgebra. This provides a wealth of examples for which we can use the computational tools of Theorem \ref{comp1}. Before stating particular results, we discuss the motivation for studying topological coHochschild homology of suspension spectra. Recall the following theorem  from Malkiewich \cite{malkiewich} and Hess--Shipley \cite{hs.coTHH}:
\begin{thm}
For $X$ a simply connected space, 
\[
\THH(\Sigma^{\infty}_+(\Omega X)) \simeq \coTHH(\Sigma^{\infty}_+X)
\]
\end{thm}

Consequently, for a simply connected space $X$, to understand $\THH$ of $\Sigma^{\infty}_+(\Omega X)$, one can instead compute the $\coTHH$ of the suspension spectrum of $X$.  
In some cases this topological coHochschild homology computation is more accessible. In particular, to compute the homology  $H_*(\THH(\Sigma^{\infty}_+(\Omega X)); k)$ using the B\"okstedt spectral sequence, the necessary input is the Hochschild homology of the algebra $H_*( \Sigma^{\infty}_+(\Omega X); k)$. To compute $H_*(\coTHH(\Sigma^{\infty}_+X); k)$ the spectral sequence input is the coHochschild homology of the coalgebra $H_*(\Sigma^{\infty}_+X; k)$. In some cases this homology coalgebra is easier to work with than the homology algebra of the loop space, making the topological coHochschild homology calculation a more accessible path to studying the topological Hochschild homology of based loop spaces. For instance, the Lie group calculations in Example \ref{Lie} are more approachable via the coTHH spectral sequence. Topological Hochschild homology of based loop spaces is of interest due to connections to algebraic $K$-theory, free loop spaces,  and string topology, which we will now recall.

Recall B\"okstedt and Waldhausen showed that $\THH(\Sigma^{\infty}_+(\Omega X)) \simeq \Sigma^{\infty}_+\mathcal{L} X$, where $\mathcal{L} X$ is the free loop space, $\mathcal{L} X = \Map(S^1, X)$. The homology of a free loop space, $H_*(\mathcal{L} X)$, is the main object of study in the field of string topology \cite{ChasSullivan, CJY}. Since 
\[\coTHH(\Sigma^{\infty}_+X)\simeq \Sigma^{\infty}_+\mathcal{L} X\]
for simply connected $X$, topological coHochschild homology gives a new way of approaching this homology.

Topological coHochschild homology of suspension spectra also has connections to algebraic $K$-theory. Waldhausen's $A$-theory of spaces $A(X)$ is equivalent to $K(\Sigma^{\infty}_+\Omega X)$. Recall that there is a trace map from algebraic $K$-theory to topological Hochschild homology
\[
A(X) \simeq K(\Sigma^{\infty}_+\Omega X) \rightarrow \THH(\Sigma^{\infty}_+\Omega X). 
\]
When $X$ is simply connected this gives us a trace map to topological coHochschild homology as well: 
\[
A(X) \simeq K(\Sigma^{\infty}_+\Omega X) \rightarrow \THH(\Sigma^{\infty}_+\Omega X) \simeq \coTHH(\Sigma^{\infty}_+X). 
\]
In the remainder of this section we will make some explicit computations of the homology of coTHH of suspension spectra. 
\begin{ex} Let  $X$ be a product of $n$ copies of $\mathbb{C}P^{\infty}$. As a coalgebra 
\[H_*(\Sigma^{\infty}_+X; k) \cong \Gamma_{k}(x_1, x_2, \ldots x_n)\]
 where $\deg(x_i) = 2$. By Theorem \ref{comp1}, there is an isomorphism of graded $k$-modules
\[
H_*(\coTHH(\Sigma^{\infty}_+X); k) \cong \Gamma_{k}(x_1, x_2, \ldots x_n) \otimes \Lambda_{k}(\hat{x}_1, \hat{x}_2, \ldots \hat{x}_n).
\] 
\end{ex}

Note that any space having cohomology with polynomial generators in even degrees will give us a suspension spectrum whose homology satisfies the conditions of Theorem \ref{comp1}. Hence this theorem allows us to compute the topological coHochschild homology of various suspension spectra, as in the example below. 

\begin{ex}\label{Lie}
Theorem \ref{comp1} directly yields the following computations of topological coHochschild homology.  The isomorphisms are of graded $k$-modules.
\begin{align*}
H_*(\coTHH(\Sigma^{\infty}_+BU(n)); k) &\cong  \Gamma_{k}(y_1, \ldots y_n) \otimes \Lambda_{k}(\hat{y}_1, \ldots \hat{y}_n) \\
&\qquad\qquad\abs{y_i} = 2i,\  \abs{\hat{y}_i} = 2i-1 \\[2ex]
H_*(\coTHH(\Sigma^{\infty}_+BSU(n)); k)& \cong  \Gamma_{k}(y_2, \ldots y_n) \otimes \Lambda_{k}(\hat{y}_2, \ldots \hat{y}_n) \\
& \qquad\qquad \abs{y_i} = 2i,\  \abs{\hat{y}_i} = 2i-1 \\[2ex]
H_*(\coTHH(\Sigma^{\infty}_+BSp(n)); k) &\cong \Gamma_{k}(z_1, \ldots z_n) \otimes \Lambda_{k}(\hat{z}_1, \ldots \hat{z}_n) \\
& \qquad\qquad\abs{z_i} = 4i,\  \abs{\hat{z}_i} = 4i-1 
\end{align*}

With some restrictions on the prime $p$, we also immediately get the mod $p$ homology of the topological coHochschild homology of the suspension spectra of the classifying spaces $BSO(2k), BG_2, BF_4,$  $BE_6, BE_7,$ and $BE_8$. Note that these computations also yield the homology of the free loop spaces $H_*(\mathcal{L}BG; \F_p)$ and the homology $H_*(\THH(\Sigma^{\infty}_+G); \F_p)$, for $G = U(n), SU(n), Sp(n), SO(2k),$ $G_2, F_4. E_6, E_7,$ and $E_8$. The topological coHochschild homology calculations also follow directly for products of copies of any of these classifying spaces $BG$.  
\end{ex}

\appendix
\section{Proof of Lemma \ref{lem:ReedyFibSurj}}\label{appendix:yuckyproof}
We prove the assertion of Lemma \ref{lem:ReedyFibSurj}.  This shows that $\coTHH^\bullet(C)$ is Reedy fibrant when $C$ is a counital coaugmented coalgebras in the category of graded $k$-modules for a commutative ring $k$.  
\begin{proof}[Proof of Lemma \ref{lem:ReedyFibSurj}]
Let $\eta\colon k\to C$ be the coaugmentation or the $k$-module map such that $\epsilon\eta=\id$.  Consider the maps $\eta_i\colon C^{\otimes n}\to C^{\otimes n+1}$ given by applying $\eta$ between the $i+1$ and $i+2$ copy of $C$; that is 
\[ \eta_i=\id^{\otimes i+1}\otimes \eta \otimes \id^{\otimes n-i-1}. \]
The indexing has been chosen to match the indexing on the codegeneracy maps $\sigma_i$ in the sense that  $\sigma_i\eta_i=\id$.  Direct calculation yields the following additional relations among the $\eta_i$ maps and codegeneracies:
\[\begin{cases}
\eta_j\eta_i=\eta_i\eta_{j-1}  & i<j\\
\sigma_j\eta_i = \eta_i\sigma_{j-1} & i<j\\
\sigma_i\eta_j =\eta_{j-1}\sigma_i & i<j
\end{cases}\]
The last two relations in particular imply that $\sigma_{i+1}\eta_i=\eta_i\sigma_i=\sigma_i\eta_{i+1}$.  To get a sense for these operations, note that if $c_0\otimes\dotsb\otimes c_n$ is a simple tensor, then
\[\eta_i\sigma_i(c_0\otimes \dotsb \otimes c_n)= c_0\otimes \dotsb \otimes c_{i}\otimes \eta\epsilon(c_{i+1})\otimes c_{i+1}\otimes \dotsb \otimes c_n\]
so that $\eta_i\sigma_i$ replaces the $i+1$th tensor factor $c_{i+1}$ with the result of sending it to $k$ via the counit and then mapping back up via the coaugmentation.

From Hirschhorn, we know that the matching space $M_n(X^\bullet)$ for a cosimplicial graded $k$-module $X^\bullet$ is given by
\[M_n(X^\bullet)=\{(x_0,\dotsc,x_{n-1}) \in (X^{n-1})^{\times n}\mid \sigma_i(x_j)=\sigma_{j-1}(x_i) \text{ for } 0\leq i<j\leq n-1\}.\]
Let $X^\bullet=\coTHH(C)^\bullet$ and consider $(x_0,\dotsc, x_{n-1})\in M_n(X^\bullet)$. We define $y\in \coTHH(C)^{n}=C^{\otimes n+1}$ to be the following sum:
\[y= \sum_{i=0}^{n-1} \eta_i\left(x_i +\sum_{j=1}^i (-1)^j\sum_{0\leq l_1<\dotsb<l_j\leq i-1} \eta_{l_1}\sigma_{l_1}\dotsm\eta_{l_j}\sigma_{l_j}(x_i)\right).\]
That is, for each $i$, we apply $\eta_i$ to $x_i$ plus a signed sum over all tuples of distinct numbers between $0$ and $i-1$ of applying $\eta\epsilon$ to $x_i$ in the spots corresponding to these numbers. The is sign determined by the number of elements in a tuple.  (In fact, we can think of $x_i$ itself as corresponding to the ``empty tuple'' and so unify our description of these terms, but that seems more obfuscating than strictly necessary.) 

We show that $\sigma_s(y)=x_s$ for all $0\leq s\leq n-1$, so that the matching map applied to $y$ yields $(x_0,\dotsc,x_{n-1})$.

First, break $\sigma_s(y)$ into terms corresponding to $i<s$, $i=s$ and $i>s$. 
\begin{align*}
\sigma_s(y)&=\sum_{i=0}^{s-1}\sigma_s\eta_i\left(x_i+\sum_{j=1}^{i} (-1)^j\sum_{0\leq l_1<\dotsb<l_j\leq i-1} \eta_{l_1}\sigma_{l_1}\dotsm\eta_{l_j}\sigma_{l_j}(x_i)\right)\\
&+\sigma_s\eta_s\left(x_s+\sum_{j=1}^s (-1)^j\sum_{0\leq l_1<\dotsb<l_j\leq s-1} \eta_{l_1}\sigma_{l_1}\dotsm\eta_{l_j}\sigma_{l_j}(x_s)\right)\\
&+\sum_{i=s+1}^{n-1}\sigma_s\eta_i\left(x_i+\sum_{j=1}^i(-1)^j\sum_{0\leq l_1<\dotsb<l_j\leq i-1}\eta_{l_1}\sigma_{l_1}\dotsm\eta_{l_j}\sigma_{l_j}(x_i)\right)
\end{align*}
which by the relationship between $\sigma_s$ and $\eta_i$ in each case we can rewrite as 
\begin{align*}
\sigma_s(y)&=\sum_{i=0}^{s-1}\eta_i\sigma_{s-1}\left(x_i+\sum_{j=1}^{i} (-1)^j\sum_{0\leq l_1<\dotsb<l_j\leq i-1} \eta_{l_1}\sigma_{l_1}\dotsm\eta_{l_j}\sigma_{l_j}(x_i)\right)\\
&+\left(x_s+\sum_{j=1}^s (-1)^j\sum_{0\leq l_1<\dotsb<l_j\leq s-1} \eta_{l_1}\sigma_{l_1}\dotsm\eta_{l_j}\sigma_{l_j}(x_s)\right)\\
&+\sum_{i=s+1}^{n-1}\eta_{i-1}\sigma_s\left(x_i+\sum_{j=1}^i(-1)^j\sum_{0\leq l_1<\dotsb<l_j\leq i-1}\eta_{l_1}\sigma_{l_1}\dotsm\eta_{l_j}\sigma_{l_j}(x_i)\right)
\end{align*}
We will show that the terms with $i>s$ all vanish and that the terms with $i\leq s$ cancel to leave only  $x_s$.

Observe that whenever $l< s$, our relations plus the codegeneracy relations imply
\begin{equation}\label{keyrelation} \sigma_s\eta_l\sigma_l=\eta_l\sigma_{s-1}\sigma_l=\eta_l\sigma_l\sigma_s.\end{equation}
That is, $\sigma_s$ ``commutes with'' the operation $\eta_l\sigma_l$ for $l<s$.

We first prove that for each $i$, $s+1\leq i\leq n-1$, the sum
\begin{equation}\label{sumforlargei}\sigma_sx_i+\sigma_s(\sum_{j=1}^{i}(-1)^j\sum_{0\leq l_1<\dotsb<l_j\leq i-1} \eta_{l_1}\sigma_{l_1}\dotsm \eta_{l_j}\sigma_{l_j}(x_i))
\end{equation}
vanishes.  In fact, this follows purely from the relations on the $\eta$'s and $\sigma$'s and doesn't depend on $x_i$ at all and we will drop the $x_i$ and just consider the operation given by our sum of $\eta$'s and $\sigma$'s. 

For each tuple $0\leq l_1<\dotsb < l_j\leq i-1$, let $l_a$ be the entry with $l_a< s \leq l_{a+1}$.  The relation of Equation (\ref{keyrelation}) lets us move the $\sigma_s$ past all the $\eta$'s and $\sigma$'s until we get to the term indexed by $l_{a+1}$:
\[ \sigma_s\eta_{l_1}\sigma_{l_1}\dotsm\eta_{l_j}\sigma_{l_j} =\eta_{l_1}\sigma_{l_1}\dotsm\eta_{l_a}\sigma_{l_a}\sigma_s\eta_{l_{a+1}}\sigma_{l_{a+1}}\dotsm\eta_{l_j}\sigma_{l_j}\]

We then split the operation in the summation term of Equation \ref{sumforlargei} into two cases: the summation over tuples where $s=l_{a+1}$ and the summation over tuples where $s<l_{a+1}$.

If $l_{a+1}=s$, then $\sigma_s\eta_{l_{a+1}}\sigma_{l_{a+1}}= \sigma_s\eta_s\sigma_s=\sigma_s$, so these tuples yield the operation
\[ -\sigma_s +\sum_{j=2}^i\sum_{\substack{\text{$j$-tuples}\\ \text{containing $s$}}} (-1)^j \eta_{l_1}\sigma_{l_1}\dotsm\eta_{l_a}\sigma_{l_a}\sigma_s\eta_{l_{a+2}}\sigma_{l_{a+2}}\dotsb\eta_{l_j}\sigma_{l_j}\]
The tuples where $l_{a+1}>s$ yield the operation
\[ \sum_{j=1}^i\sum_{\substack{\text{$j$-tuples not}\\ \text{containing $s$}}} (-1)^j\eta_{l_1}\sigma_{l_1}\dotsb\eta_{l_a}\sigma_{l_a}\sigma_s\eta_{l_{a+1}}\sigma_{l_{a+1}}\dotsm\eta_{l_j}\sigma_{l_j}.\]
When we sum these two operations, all the terms cancel except $-\sigma_s$ because every $j$-tuple containing $s$ yields a $j-1$-tuple not containing $s$ after omitting the $s$.   Thus  
\[\sigma_s(\sum_{j=1}^{i}(-1)^j\sum_{0\leq l_1<\dotsb<l_j\leq i-1} \eta_{l_1}\sigma_{l_1}\dotsm \eta_{l_j}\sigma_{l_j}(x_i)) =-\sigma_s(x_i)
\]
and each term of $\sigma_s(y)$ coming from $i>s$ vanishes.

Our next task is to understand the terms of the form
\[\eta_i\sigma_{s-1}\left(x_i+\sum_{j=1}^i \sum_{0\leq l_1<\dotsb<l_j\leq i-1} \eta_{l_j}\sigma_{l_1}\dotsm \eta_{l_j}\sigma_{l_j}(x_i)\right)\]
when $i<s$.  Since in this case, each entry of the tuple is less than $s-1$, we use the relation in Equation (\ref{keyrelation}) to move the $\sigma_{s-1}$ all the way to the right:
\[\eta_i\sigma_{s-1}(x_i)+\sum_{j=1}^i \sum_{0\leq l_1<\dotsb<l_j\leq i-1} \eta_i\eta_{l_j}\sigma_{l_1}\dotsm \eta_{l_j}\sigma_{l_j}\sigma_{s-1}(x_i)
\]
Each entry of the tuple is also less than $i$, so the  additional relation $\eta_i\eta_l\sigma_l=\eta_l\eta_{i-1}\sigma_l=\eta_l\sigma_l\eta_{i}$  when $l<i$ allows us move the $\eta_i$ to the right and we obtain
\[\eta_i\sigma_{s-1}(x_i)+\sum_{j=1}^i \sum_{0\leq l_1<\dotsb<l_j\leq i-1} \eta_{l_j}\sigma_{l_1}\dotsm \eta_{l_j}\sigma_{l_j}\eta_i\sigma_{s-1}(x_i)
\]
and since $(x_0,\dotsc,x_{n-1})$ is in the matching space $M_n(X^\bullet)$, $\sigma_{s-1}(x_i)=\sigma_i(x_s)$.

Putting this all together we find:
\begin{align*}
\sigma_s(y)&=\sum_{i=0}^{s-1} \eta_i\sigma_i(x_s) +\sum_{j=1}^i \sum_{0\leq l_1<\dotsb<l_j\leq i-1} (-1)^j\eta_{l_1}\sigma_{l_1}\dotsm\eta_{l_j}\sigma_{l_j}\eta_i\sigma_{i}(x_s)\\
&\qquad + x_s +\sum_{j=1}^s \sum_{0\leq l_1<\dotsb<l_j\leq s-1}(-1)^j\eta_{l_1}\sigma_{l_1}\dotsm\eta_{l_j}\sigma_{l_j}(x_s)\\
&=\sum_{i=0}^{s-1}\sum_{j=1}^{i}\sum_{0\leq l_1<\dotsb <l_j\leq i-1}(-1)^j \eta_{l_1}\sigma_{l_1}\dotsm \eta_{l_j}\sigma_{l_j}\eta_i\sigma_i(x_s)\\
&\qquad +x_s+\sum_{j=2}^s\sum_{0\leq l_1<\dotsb<l_j\leq s-1} \eta_{l_1}\sigma_{l_1}\dotsm\eta_{l_j}\sigma_{l_j}(x_s)
\end{align*}
For fixed $i$, the terms in the first summation correspond to the tuples in the second summation whose final term is $l_j=i$.  Hence the two summations run over the same tuples but with different signs, so they cancel and $\sigma_s(y)=x_s$ as required.  
\end{proof}

\bibliography{coTHH}
\bibliographystyle{plain}

\end{document}